\documentclass{amsart}
\usepackage{amsmath}
\usepackage{amssymb}
\usepackage{amsthm}
\usepackage{mathrsfs}
\usepackage{comment}
\usepackage{hyperref}

\usepackage[all,cmtip]{xy}\usepackage{xcolor}
\usepackage{enumerate}
\usepackage{bm}
\usepackage{dsfont}
\usepackage{mathtools}
\usepackage[cal=euler]{mathalfa}
\usepackage[top=1in, bottom=1.25in, left=1.25in, right=1.25in]{geometry}
\usepackage{parskip}

\hypersetup{colorlinks=true,linkcolor=magenta,citecolor=blue}

\DeclareMathAlphabet{\mathpzc}{OT1}{pzc}{m}{it}

\theoremstyle{plain}
\newtheorem{theorem}{Theorem}[section]
\newtheorem*{theorem*}{Theorem}

\newtheorem{lemma}[theorem]{Lemma}
\newtheorem*{claim*}{Claim}
\newtheorem{proposition}[theorem]{Proposition}

\newtheorem{corollary}[theorem]{Corollary}

\theoremstyle{definition}

\newtheorem{remark}[theorem]{Remark}

\numberwithin{equation}{section}
\numberwithin{figure}{section}

\newcommand{\N}{\mathbb{N}}

\DeclareMathOperator{\Gal}{Gal}
\DeclareMathOperator{\Aut}{Aut}

\newcommand{\ignore}[1]{}

\author{Susanne Pumpl\"un}
\address{School of Mathematical Sciences, University of Nottingham,
Nottingham NG7 2RD}
\email{Susanne.Pumpluen@nottingham.ac.uk}



\date{\today}

\begin{document}
    \title{Explicit constructions of anti-automorphisms of cyclic and generalized cyclic algebras}
    \maketitle

\begin{abstract}
We present norm criteria for the existence of anti-automorphisms, as well as explicit constructions of anti-automorphisms, both on cyclic and generalized cyclic algebras. Our approach describes anti-automorphisms as polynomial maps and unifies existing approaches. It recovers classical criteria for the existence of involutions as special cases.

 We obtain  norm conditions for the existence of anti-automorphisms of the second kind 
 on the ring of twisted Laurent series $K((t;\sigma))$ over a field $K$ and the ring of twisted Laurent series $D((t;\sigma))$  over a division algebra $D$ that is finite-dimensional over its center.

 Our constructions rely on the isomorphisms between an algebra and its opposite algebra. Along the way, we hence describe monomial isomorphisms between cyclic or generalized cyclic algebras,
  which ties in with studying the isomorphism problem for central simple algebras. Proper nonassociative cyclic and generalized cyclic algebras,  which are canonical generalizations of associative cyclic and generalized cyclic central simple algebras, are included here.
\end{abstract}

    \section{Introduction}

The existence and classification of involutions is a classical question in the theory of central simple algebras. Involutions on central simple algebras either preserve the center (involutions of the first kind) or act on the center nontrivially (involutions of the second kind). Their study involves for instance Brauer groups, and quadratic and hermitian forms \cite{KMRT}.

 It is well-known that a central simple algebra  over $F$ admits an involution of the second kind which leaves a subfield $F_0$ of $F$ fixed, if and only if its Brauer class lies in the kernel of the corestriction map
${\rm Cor}_{F/F_0}:\mathrm{Br}(F)\to \mathrm{Br}(F_0).$ For cyclic algebras $(K/F,\sigma,a)$ in particular, we thus know that
${\rm Cor}_{F/F_0}([(K/F,\sigma,a)])=[(K/F,\sigma,a\tau(a))],$ where $\tau$ generates the automorphism group of $F/F_0$,
and consequently, $(K/F,\sigma,a)$ admits an involution of the second kind, if and only if
$a\tau(a)\in N_{K/F}(K^\times).$ We call such a condition  a \emph{norm condition}.

While the existence of involutions (anti-automorphisms of order two)  on central simple algebras  is well understood and has been studied extensively, much less is known about general anti-automorphisms of central simple algebras,  see \cite{Lewis, Lewis2, LewisTig} for some results. Explicit examples of central simple division algebras
 over algebraic number fields $F$ which have anti-automorphisms of the second kind, but no involution can be found in \cite{MST}.

We will distinguish between anti-automorphisms of a central simple algebra $A$ over a field $F$ which are $F$-linear (of the first kind), and $F_0$-linear (of the second kind), where $F_0$ is a proper subfield of $F$. More precisely, we will investigate anti-automorphisms of the second kind that are  $\tau$-semilinear for some $\tau\in \Aut(F)$ with fixed field $F_0=F^\tau$. Unlike in the case of involutions, for anti-automorphisms of the second kind the  extension $F/F_0$ will generally not be quadratic \'etale. Also note that our definition of anti-automorphisms of the second kind differs from the one presented in \cite{Lewis2}.

 We will focus on anti-automorphisms  both of cyclic algebras and of generalized cyclic algebras as defined in \cite{J96} (note also the slightly different approach described in \cite{T}). Generalized cyclic algebras do not need to be crossed products, e.g. see \cite{TH, TH2004, TH2005}. However, every solvable crossed product is isomorphic to a generalized cyclic algebra by a classical result of Albert \cite[p.~182-187]{A}.

Our approach is based on describing cyclic algebras as quotient rings  $K[t;\sigma]/K[t;\sigma](t^n-a)$, respectively generalized cyclic algebras  as quotient rings $D[t;\sigma]/D[t;\sigma](t^m-d)$ with $D$ a central simple algebra, employing skew polynomials. We show that all anti-automorphisms of a cyclic algebra $(K/F,\sigma,a)$ that restrict to some anti-automorphism on $K$, and under some mild conditions also all anti-automorphisms of a generalized cyclic division algebra $(D,\sigma,d)$ that restrict to some anti-automorphism on  $D$, are canonically induced by anti-automorphisms of the skew polynomial rings  $K[t;\sigma]$, respectively  $D[t;\sigma]$. They map $t$ to $\alpha t^k$ for some suitable positive integer $k$ and $\alpha\in F^\times$ (respectively in $C^\times$). We call an anti-automorphisms (resp., an isomorphism) mapping $t$ to $\alpha t^k$ a \emph{monomial anti-automorphism (resp., isomorphism) of degree} $k$. Our monomial anti-automorphisms of degree $k$ are in one-one correspondence with isomorphisms between the algebra and its opposite algebra which are monomial of degree $m-k$.

In order to describe the anti-automorphisms we also obtain new results on how to write certain classes of isomorphisms between (potentially nonassociative) cyclic algebras and between (potentially nonassociative)  generalized cyclic algebras.  We show that monomial isomorphism
of degree $k>1$ that extend $\tau$ can only exist between associative cyclic algebras  (Theorem \ref{general_isomorphism_theorem2}), and   isomorphism between proper nonassociative cyclic algebras $(K/F,\sigma,a)$ and $ (K/F,\sigma,b)$
 that extend $\tau$ will be monomial of degree one and require that  $\tau$ commutes with $\sigma$.

More precisely, suppose that $\tau\in \Aut(K)$ and $\tau \sigma\tau^{-1}=\sigma^k$ for some $k\in \{1,\dots, n-1 \}$ such that $\gcd (k, n)=1$.
 Let $\alpha\in K^\times$. Then there exists a ring isomorphism $(K/F,\sigma,a)$ and $ (K/F,\sigma,b)$ extending $\tau$ that maps $t$ to $\alpha t^k$ for  $k>1$
  if and only if $(K/F,\sigma,a)$ is associative and
 $$N_{K/F}(\alpha) b^k=\tau(a)$$
 (Corollary  \ref{cor:general_isomorphism_theorem2}).

Monomial  anti-automorphisms of degree $k>1$ greater than one that extend $\tau$ do not exist on associative cyclic algebras when
$\tau\in \Aut(K)$ commutes with $\sigma$ (Theorem \ref{thm:mainassociative cyclic1}).
For associative cyclic algebras $(K/F,\sigma,a)$, and when $\tau$ commutes with $\sigma$, only monomial anti-automorphisms of degree one exist that extend $\tau$, additionally assuming there is some $\alpha\in K^\times$ such that
$N_{K/F}( \alpha)= \tau(a)a$  (Theorem \ref{thm:mainassociative cyclic1}). Monomial anti-automorphisms of degree $k$ for $k\in \{2,\dots, n-1 \}$  with $\gcd(k,n)=1$ are possible, whenever there exist $\tau\in \Aut(K)$ such that
 $\tau \sigma\tau^{-1}=\sigma^k,$
and if, additionally,
  $N_{K/F}(\alpha)=\tau(a)a^{k}$ (Theorem  \ref{general_antiautomorphism_theorem4}).

The situation for generalized cyclic algebras is similar, with some small modifications (e.g., Theorem \ref{thm:maingeneralizedantiautcor}, Theorem \ref{cor:tilde G}). Most notably,  when $\sigma$ has order $m$,  monomial anti-automorphisms of degree greater than one can exist on associative generalized cyclic algebras. 

The conditions for the existence of anti-automorphisms again involve a norm condition (Theorem \ref{cor:tilde G}).

The existence of monomial isomorphisms of degree $k$ between generalized cyclic algebras is treated along the way and yields some straightforward contributions to the isomorphism problem for generalized cyclic algebras, e.g.:

\begin{theorem*} (Theorem \ref{thm:autcor})
Let $D$ be a central division algebra over $C$ of degree $r$, and let $(D, \sigma, d)$ and $(D', \sigma', d)$ be two associative  generalized cyclic  algebras. Let $\tau:D\to D'$ be a ring isomorphism and assume that $\tau\circ\sigma=\sigma'\circ \tau$ and that $\tau$ is  $F_0$-linear, where $F_0$ is a subfield of $F$. Let  $\alpha\in C^\times$. If $G_{\tau,\alpha,k}:(D, \sigma, d)\to (D, \sigma', d')$  is an $F_0$-algebra isomorphism  for some integer $k$, $1\leq k\leq m-1$, then $k=1$; and
   $ G_{\tau,\alpha,1}:(D, \sigma, d)\to (D', \sigma', d'),$ 
   $$G_{\tau,\alpha,1}(\sum_{i=0}^{m-1}d_it^i)=\sum_{i=0}^{m-1}\tau(d_i) (\alpha t)^i= \sum_{i=0}^{m-1}\tau(d_i) N_i^{\sigma'}(\alpha) t^i $$
    is an $F_0$-algebra isomorphism if and only if
$$\tau(d) =N^{\sigma'}_{m}(\alpha) d'.$$
\end{theorem*}

This paper grew out of our attempt to understand  the proof of \cite[Proposition 3.4]{MST}. This proposition displays an elegant example of a cyclic division algebra over a number field, that admits no involution but carries an anti-automorphism of the second kind and lays the groundwork for the main result \cite[Theorem 4.4]{MST} which contains yet another elegant construction of such a scenario, this time for a twisted Laurent series ring $K((t;\sigma))$. Both proofs explicitly write down some monomial anti-automorphisms of the second kind and of degree one. Both results are studied in Section \ref{sec:Pat}
(Theorem \ref{thm:Patmain}, Proposition \ref{prop:Pat}).
 We give criteria for the existence of anti-automorphisms of infinite order on the twisted ring of Laurent polynomials $K((t;\sigma))$, a cyclic division algebra over $F((t^n))$ (Theorem \ref{thm:Patrelvant}), and for the existence of anti-automorphisms of infinite order on the twisted ring of Laurent polynomials $D((t;\sigma))$ for a division algebra $D$ (Theorem \ref{thm:maingeneralizedantiaut6}), which is a generalized cyclic division algebra.

Our approach can be easily generalized to anti-automorphisms of Azumaya algebras which is briefly discussed in Section \ref{sec:Azumaya}.

\section{Preliminaries}

\subsection{Nonassociative rings and algebras} \label{subsec:2}

Let $S$ be a unital commutative ring and let $A$ be an $S$-module.
We call $A$ a \emph{(nonassociative) algebra} over $S$ if there exists an $S$-bilinear
map $A\times A\mapsto A$, $(x,y) \mapsto x \cdot y$, the  \emph{multiplication} of $A$. We will  denote this multiplication by juxtaposition. We will only consider \emph{unital} algebras, which possess an element denoted by
1, such that $1x=x1=x$ for all $x\in A$.

  The associativity in $A$ is measured by the {\it left nucleus} defined as ${\rm Nuc}_l(A) = \{ x \in A \, \vert \, [x, A, A]
= 0 \}$, the {\it middle nucleus} ${\rm Nuc}_m(A) = \{ x \in A \,
\vert \, [A, x, A]  = 0 \}$ and  the {\it right nucleus}   ${\rm
Nuc}_r(A) = \{ x \in A \, \vert \, [A,A, x]  = 0 \}$, where $[x, y, z] = (xy) z - x (yz)$.
Their intersection
 ${\rm Nuc}(A) = \{ x \in A \, \vert \, [x, A, A] = [A, x, A] = [A,A, x] = 0 \}$ is the {\it nucleus} of $A$.
The  {\it commuter} of $A$ is defined as ${\rm
Comm}(A)=\{x\in A\,|\,xy=yx \text{ for all }y\in A\}$ and the {\it
center} of a nonassociative algebra $A$ is defined as ${\rm C}(A)=\text{Nuc}(A)\cap  {\rm Comm}(A)$.

Every unital nonassociative algebra is a unital nonassociative ring $(A,+,\cdot)$ with the distributivity laws satisfied by definition of its multiplication. Conversely, every nonassociative ring can be viewed as a nonassociative algebra over its center, or any subring of it.

 A nonassociative ring (or algebra) $A$  is called a \emph{proper} nonassociative ring (or algebra) if it is not associative.

Write $A^{op}=\{x^{op}\,|\, x\in A\}$ for the opposite algebra  of $A$. We recall that $A^{op}$ has the same addition as $A$ and multiplication given by
 $x^{op}\circ y^{op}=(yx)^{op}$.

Let $D$ be  a central simple algebra over some field and let $R=D[t;\sigma]$ be the ring of skew polynomials, with $\sigma\in \Aut(D)$. We will look at the algebras $ \mathbb{S}_{t^m-a}=D[t;\sigma]/D[t;\sigma](t^m-a)$ which are defined on the set $R_m=\{g\,|\, \deg (g)<m\}$ of skew polynomials of degree less than $m$.  Their multiplicative structure is given by
 $$g\circ h=gh \,\,{\rm mod}_r (t^m-a) $$
 for every $g$, $h$ of degree less than $m$,
 where the right-hand side denotes the remainder of $gh$ after right division with $t^m-a$ \cite{P66}. These algebras are the well-known associative quotient algebras when $t^m-a$ generates a two-sided ideal in $D[t;\sigma]$, and are the lesser known proper nonassocative Petit algebras otherwise.
 
 Note that we also have a second algebra structure on $R_m$, this time given by left division by $f$;
 $$g\circ h=gh \,\,{\rm mod}_l (t^m-a) $$
 for every $g$, $h$ of degree less than $m$. The resulting algebra, denoted $ \,_{t^m-a} \mathbb{S}=D[t;\sigma]/(t^m-a)D[t;\sigma]$, satisfies
 $( \mathbb{S}_{t^m-a})^{op}\cong \,_{t^m-a} \mathbb{S}$.

\subsection{Anti-automorphisms} Let $A$ be a (potentially nonassociative) ring. An \emph{anti-homomorphism} on $A$ is an additive map $f:A\to A$ such that
$f(xy)=f(y)f(x)$ for all $x,y\in A$. If $f$ is also bijective then $f$ is called an  \emph{anti-automorphism} on $A$. If we read $A$ as  an $S$-algebra then the restriction of an anti-automorphism $f:A\to A$  to $S$ is an automorphism $\tau=f|_S\in \Aut(S)$. An anti-automorphism $f$ is called \emph{of the first kind} if it is $S$-linear, and \emph{of the second kind} otherwise. (This deviates from the definition in
 \cite[p.~254]{Lewis}, where anti-automorphisms of the second kind are defined to have order two when restricted to $S$.)
 Put $S_0={\rm Fix}(\tau)=S^\tau$, then $S_0$ is a subring of $S$.
 If $\tau=f|_S\in \Aut (S)$ is non-trivial  we call $f$ a \emph{$\tau$-semilinear } anti-automorphism, if we want to make clear which $\tau$ is involved. Every $\tau$-semilinear anti-automorphism is an $S_0$-linear anti-automorphism.

We note that odd powers of an anti-automorphism  $f$ are again anti-automorphisms and even powers are isomorphisms.

When $A$ is a central simple associative algebra over a field $F$  and $f$ is an involution on $A$ of the second kind with $F_0={\rm Fix}(f|_F)$, we know that $F/F_0$ is always a quadratic \'etale  field extension with Galois group generated by $\tau$.  When $F/F_0 $ is a Galois field extension and $D$  a central division
$F$-algebra such that $D \not\cong D^{op}$ we also know that if every element of order two in ${\rm Gal}(F/F_0)$ extends to
an automorphism of $D$, then $D$ has no involution which restricts to the identity on
$F_0$. In particular, if $F_0$ has no nontrivial automorphism (e.g. $F_0 = \mathbb{Q}$), then $D$ has
no involution at all \cite[Proposition 3.2.]{MST}.

By \cite[Lemma 3.1]{MST}, if $D$ is a noncommutative division ring with an anti-automorphisms $f$ then $f^\ell=id$ implies that $\ell$ must be even and if $D$ has no involution it must be a power of 4.

 By a well-known Theorem of Albert, every finite dimensional central division algebra $D$ with
an $F$-linear anti-automorphism has an involution \cite[Theorem 7]{Lewis}. Albert proved that a central simple $F$-algebra $A$ admits an involution of the first kind if and only if it admits an anti-automorphism of the first kind, and that this in turn is equivalent to $A$ and $A^{op}$ being isomorphic $F$-algebras.

The composite of two anti-automorphism $f,g:A\to A$   is an automorphism of $A$, hence there exists some $u\in A^\times$ such that
$f\circ g(x)=uxu^{-1}$ which implies that $f$ and $g$ differ by an inner automorphism whenever $A$ is a central simple algebra over a field $F$. When $A$ is a central simple algebra over $F$ and both $f$ and $g$ are $F$-linear involutions, then $u$ satisfies $f(u)=\pm u$, whereas if both are involutions of the second type we may chose $u\in A$ such that $f(u)=u$, e.g. see \cite{KMRT}.

For an associative cyclic $S$-algebra $A$, we
 write  $A^{\tau}=\{x^{\tau}\,|\, x\in A\}$ for the conjugate algebra
 $A^{\tau}$, for $\tau\in \Aut{S}$. We recall that  $A^{\tau}$ has the same addition and multiplication as $A$ and scalar multiplication given by twisting the scalar multiplication via $a x^{\tau}=(\tau^{-1}(a)x)^{\tau}$  for all $a\in S$.

It is well-known that a map $\widetilde f:A\to A$ is a ring anti-automorphism of an associative $F$-algebra $A$ which restricts to some $\tau\in \Aut (F)$, if and only if the map $f:A^\tau\to A^{op}$
 given by $f(x^\tau)=\widetilde f(x)^{op}$ is an isomorphism of $F$-algebras \cite[p. 2]{MST}.
 This holds for algebras defined over rings a well, but we will not use this point of view.

  We will take a slightly different approach. Given some $\tau\in \Aut_{S_0}(S)$, there is an obvious one-one correspondence between $\tau$-semilinear anti-automorphisms
  $\widetilde f:A\to A$ and  $\tau$-linear ring isomorphisms  $f:A\to A^{op}$ (which can be also viewed as algebra isomorphisms between $S_0$-algebras) defined via $f(x)=\widetilde f(x)^{op}$.

We will study isomorphisms not just for associative algebras, but also for more general nonassociative unital Petit algebras  $K[t;\sigma]/K[t;\sigma](t^m-a)$ and $D[t;\sigma]/D[t;\sigma](t^n-a)$, e.g. where $a\in K^\times$ and $m$ might not be the order of $\sigma $ \cite{P66}.
 By \cite[(15)]{P66}, the algebras $K[t;\sigma]/K[t;\sigma](t^n-a)$,  for example, are not associative exactly when $\sigma(a)\neq a$ or $n\nmid m$.

\subsection{Cyclic algebras}

Let $K/F$ be a cyclic Galois field extension of degree $n$ with $\Gal (K/F)=\langle \sigma \rangle$. Our approach will be based on viewing cyclic algebras as quotient algebras of skew polynomials.  Let $(K/F,\sigma,a)$ be an associative cyclic algebra over $F$, then
$(K/F,\sigma,a)=K[t;\sigma]/K[t;\sigma](t^n-a)$, where $t^n-a\in K[t;\sigma]$ is a skew polynomial that is two-sided, that is generates a two-sided ideal we factor out.

More generally, we can define a \emph{nonassociative cyclic algebra}  $(K/F,\sigma,a)=K[t;\sigma]/K[t;\sigma](t^n-a)$ \emph{of degree} $n$ as a special type of  Petit algebra, for all $a\in K^\times$ \cite{P66}. This algebra is not associative if and only if $a\in K\setminus F$ (e.g. \cite{S12}, where the opposite algebras of cyclic algebras are studied extensively, or \cite{Pum2021}), and the usual classical simple cyclic algebra otherwise. We will treat both cases simultaneously where possible, referring to (associative or proper nonassociative) ``cyclic algebras'' when doing so  in the following.

The canonical anti-isomorphism $K[t;\sigma]\to K[t;\sigma^{-1}], $ $\sum_{i=0}^\ell d_it^i \mapsto \sum_{i=0}^\ell t^i d_i $, induces an anti-isomorphism between the Petit algebras $\mathbb{S}_{t^m-a}$ and $\,_{t^m-a }\mathbb{S}$;
$$K[t;\sigma]/ (t^m-a) K[t;\sigma]\to K[t;\sigma^{-1}]/ (t^m-a) K[t;\sigma^{-1}],\quad  \sum_{i=0}^{m-1} d_it^i \mapsto \sum_{i=0}^{m-1} t^i d_i = \sum_{i=0}^{m-1} \sigma^{-i}(d_i)t^i  ,$$
which shows that $(\mathbb{S}_{t^m-a})^{op}\cong \,_{t^m-a }\mathbb{S}$:
$$\big(K[t;\sigma]/ K[t;\sigma](t^m-a)\big)^{op} \cong K[t;\sigma^{-1}]/ (t^m-a) K[t;\sigma^{-1}].$$
In particular, for an associative cyclic algebra $(K/F,\sigma,a)$, we have
$$(K/F,\sigma, a)^{op}\cong (K/F,\sigma^{-1},a), \quad
\sum_{i=0}^{m-1} d_it^i \mapsto \sum_{i=0}^{m-1} \sigma^{-i}(d_i)t^i  .$$
Moreover, when $a\in {\rm Fix}(\sigma)^\times$ then 
$t$ has a well-defined inverse $t^{-1}$ in $K[t;\sigma^{-1}]/ K[t;\sigma^{-1}](t^m-a)$ given by $t^{m-1}a^{-1}$, and
$$ K[t;\sigma]/ K[t;\sigma](t^m-a^{-1}) \to  K[t;\sigma^{-1}]/ K[t;\sigma^{-1}](t^m-a)
,\quad  \sum_{i=0}^{m-1} d_it^i \mapsto \sum_{i=0}^{m-1}d_i t^{-i} ,$$
defines an $F$-algebra isomorphism. For better readability, we will use the notation $t^{-1}$, using that
reading $t^{-i}$ as short for $t^{m-i} a^{-1}$ and
using that $t^{-i}d_i=\sigma^{-i}(d_i)t^{-i}$.
In particular, for an associative cyclic algebra $(K/F,\sigma,a)$, we therefore have
$$(K/F,\sigma, a^{-1})\cong (K/F,\sigma^{-1},a).$$
We conclude that for associative cyclic algebras, we have
$$(K/F,\sigma, a)^{op}\cong (K/F,\sigma,a^{-1}),  \quad
\sum_{i=0}^{m-1} d_it^i \mapsto \sum_{i=0}^{m-1} \sigma^{-i}(d_i)t^{-i} .$$

\begin{remark}
Let $K/F$ be a Galois field extension of degree $n$ with automorphism group generated by $\sigma$ and suppose that $\tau\in \Aut(K)$ commute with $\sigma$. By \cite[Theorem  4.4]{Pum2025}, the map
$$G_{\tau,\alpha,1}:K[t;\sigma]/K[t;\sigma](t^m-a)\to K[t;\sigma]/K[t;\sigma](t^m-b)$$
 defined via
$$ G_{\tau,\alpha,1}(\sum_{i=0}^{m-1}d_it^i)=\sum_{i=0}^{m-1}\tau(d_i) (\alpha t)^i= \sum_{i=0}^{m-1}\tau(d_i) N_i^\sigma(\alpha) t^i,$$
is a ring isomorphism between the (potentially  nonassociative  rings) $K[t;\sigma]/K[t;\sigma](t^m-a)$ and $K[t;\sigma]/K[t;\sigma](t^m-b)$ if and only if
 $$G_{id,\alpha}: K[t;\sigma]/K[t;\sigma](t^m-a)^\tau\to K[t;\sigma]/K[t;\sigma](t^m-b)$$
  is a   ring isomorphism.
  This corresponds to the point of view employing conjugate algebras in \cite[p. 2]{MST}. Note that here  $(K[t;\sigma]/K[t;\sigma](t^m-a))^\tau = K[t;\sigma]/K[t;\sigma](t^m-\tau(a))$.
  \end{remark}

Let $(K/F,\sigma,a)=K[t;\sigma]/K[t;\sigma](t^n-a)$ be a cyclic algebra of degree $n$ over $F$.
Define
$$
{\rm Aut}(K)_\sigma=\{\tau\in {\rm Aut}(K) \mid \tau \sigma = \sigma \tau\}.
$$

\begin{lemma}\label{le:degreek}
  (i)  The elements of the group ${\rm Aut}(K)_\sigma$ are
     automorphisms of $K$ that preserve the subfield $F$ (that is, that restrict to automorphisms of $F$).
     \\ (ii) For every automorphism $\tau$ of $K$ that restricts to an automorphism of $F$ there exists a positive integer $k$, $\gcd(k,n)=1$, such that $\tau\sigma\tau^{-1}=\sigma^k$.
\end{lemma}

\begin{proof}
$(i)$ Since $\Aut(K)_\sigma$ is the centralizer subgroup ${\rm Cent}_{\Aut(K)}(\sigma)$, it is a group.  As $\sigma$ generates $\Gal(K/F)$, it follows that for any $a\in K$, we have $a\in F$ if and only if $\sigma(a)=a$.  Now suppose $\tau\in \Aut(K)_\sigma$.  Then for any $a\in K$, we have
$$
\tau(a) = \tau(\sigma(a))=\sigma(\tau(a)),
$$
and therefore $a\in F$ if and only if $\tau(a)\in F$, so $F$ is preserved by $\tau$.
\\ $(ii)$  Suppose $\tau\in \Aut(K)$  restricts to an automorphism of $F$ then  $\tau(\sigma(\tau^{-1}(a)))=a$ for all $a\in F$, hence $\tau\sigma\tau^{-1}\in \Gal (K/F),$ so there exists an integer $k$ with $\gcd(k,n)=1$, such that
$\tau\sigma\tau^{-1}=\sigma^k$.
\end{proof}

For any $i\in \mathbb{N}$, $\tau\in \Aut(K)$ and $\beta\in K$, write
$$
N_i^{\tau}(\beta) =  \tau^{i-1}(\beta) \cdots \tau(\beta) \beta.
$$
Note that $\tau(N_i^\tau(\beta)) \beta =N_{i+1}^\tau(\beta)$ and thus
\begin{equation}\label{E:Normrelation}
N_{i+j}^\tau(\beta)=  \tau^{i}(N_j^\tau(\beta))\cdot N_{i}^\tau(\beta)
\end{equation}
for all $i,j\geq 0$.  Since $K/F$ is a Galois field extension with cyclic Galois group generated by $\sigma$,  $N_{n}^\sigma=N_{K/F}$.

Let  $\tau \in {\rm Aut}(K)$ and let $F_0=F^\tau$ be the fixed field of $\tau|_F$.
 By abuse of notation, for $\tau \in {\rm Aut}(K)$ we denote $\tau|_F$ by $\tau$ as well.

\section{Monomial anti-automorphisms on associative cyclic algebras from monomial isomorphisms of degree one}

In order to classify anti-automorphisms of an associative cyclic algebra $(K/F,\sigma,a)$ of degree $n$, we will use the ring isomorphisms between
$(K/F,\sigma,a)$ and its opposite algebra $(K/F,\sigma,a)^{op}=(K/F,\sigma^{-1},a)\cong (K/F,\sigma,a^{-1})$.

 Let $K/F$ be a cyclic Galois field extension of degree $n$ with $\Gal (K/F)=\langle \sigma \rangle$. We will study the anti-automorphisms of the associative algebras $K[t;\sigma]/K[t;\sigma](t^m-a)$ first.  Thus we do not require that $n=m$ immediately. The algebra $K[t;\sigma]/K[t;\sigma](t^m-a)$ is  associative, if and only $\sigma(a)=a$ and $n \mid m$. Note that  the opposite algebra of $K[t;\sigma]/K[t;\sigma](t^m-a)$ is isomorphic to $K[t;\sigma]/K[t;\sigma](t^m-a^{-1})$.

 In this section we focus on the case that $\tau$ and $\sigma$ commute, i.e. assume $\tau \in {\rm Aut}(K)_\sigma$. This means we can employ results on isomorphisms from \cite{Pum2025}.

  We will see that the existence of $F_0$-linear anti-automorphisms of $(K/F,\sigma,a)$ that extend $\tau$ and map $t$ to $\alpha t$, so are \emph{monomial of degree one}, or map $t$ to $\alpha t^k$, so are \emph{monomial of degree $k$}  for some suitable invertible $\alpha\in K^\times$, depends on a norm condition being satisfied.

\begin{theorem}\label{thm:main0}
 Let   $\tau \in {\rm Aut}(K)_\sigma$ and let $K[t;\sigma]/K[t;\sigma](t^m-a)$ be an  associative quotient algebra.
\\ (i) There exists a $\tau$-semilinear (thus $F_0$-linear) anti-automorphism $\widetilde G_{\tau,\alpha,1}$ on $K[t;\sigma]/K[t;\sigma](t^m-a)$ defined via
$$\widetilde G_{\tau,\alpha,1}(\sum_{i=0}^{m-1}d_it^i)=\sum_{i=0}^{m-1}(\alpha t)^{-i} \tau(d_i) = \sum_{i=0}^{m-1}\sigma^{-i}(\tau(d_i) ) N_i^{\sigma^{-1}}(\alpha) t^{-i},$$
 if and only if  there exists  some $\alpha\in K^\times$ such that
$$N_{m}^\sigma( \alpha)= \tau(a)a.$$
(ii) Every $\tau$-semilinear  anti-automorphism on $K[t;\sigma]/K[t;\sigma](t^m-a)$ which restricts  to  $\tau$ has the form $\widetilde G_{\tau,\alpha,k} $ defined via
$$\widetilde G_{\tau,\alpha,k}(\sum_{i=0}^{m-1}d_it^i)= \sum_{i=0}^{m-1} (\alpha t^k)^{-i}\tau(d_i) =\sum_{i=0}^{m-1} \sigma^{-i}(\tau(d_i) ) N_i^{\sigma^{-k}}(\alpha) (t^{k})^{-i}$$
for some integer $k<m$ such that $\gcd(k,m)=1$, $k\equiv 1 \mod n$, and some $\alpha\in F^\times$ satisfying $$\tau(a) a^{k}=(N_{K/F}(\alpha))^{m/n}.$$
\end{theorem}

Note that the anti-automorphism $\widetilde G_{\tau,\alpha,k}$ is monomial of degree $m-k$, since $t^{-k}=a^{-1}t^{m-k}$.

\begin{proof}
$(i)$
Let $\alpha\in K^\times$ and $\tau \in {\rm Aut}(K)_\sigma$, as well as $a\in K^\times$. Then
   $$ G_{\tau,\alpha,1}: K[t;\sigma]/K[t;\sigma](t^m-a)\to K[t;\sigma]/K[t;\sigma](t^m-a^{-1}),$$
   $$G_{\tau,\alpha,1}(\sum_{i=0}^{m-1}d_it^i)=\sum_{i=0}^{m-1}\tau(d_i) (\alpha t)^i= \sum_{i=0}^{m-1}\tau(d_i) N_i^\sigma(\alpha) t^i$$
    is an isomorphism of $F^\tau$-algebras if and only if
$$
N_{m}^\sigma( \alpha)= a\tau(a) $$
\cite[Theorem 4.1]{Pum2025}.

This isomorphisms canonically induces the isomorphism
 $$ F_{\tau,\alpha,1}: K[t;\sigma]/K[t;\sigma](t^m-a)\to K[t;\sigma]/(t^m-a) K[t;\sigma^{-1}],$$
   $$F_{\tau,\alpha,1}(\sum_{i=0}^{m-1}d_it^i)=\sum_{i=0}^{m-1}\tau(d_i) (\alpha t)^{-i}= \sum_{i=0}^{m-1}\tau(d_i) N_i^{\sigma^{-1}}(\alpha) t^{-i}$$
between the algebra and its opposite algebra,  which again  is an isomorphism of $F^\tau$-algebras.

 The isomorphism $F_{\tau,\alpha,1}$ is in one-one correspondence with the  $\tau$-semilinear anti-automorphism
 $\widetilde G_{\tau,\alpha,1}:K[t;\sigma]/K[t;\sigma](t^m-a) \to K[t;\sigma]/K[t;\sigma](t^m-a)$,
 $$\widetilde G_{\tau,\alpha,1}(\sum_{i=0}^{m-1}d_it^i)=\sum_{i=0}^{m-1}(\alpha t)^{-i} \tau(d_i) = \sum_{i=0}^{m-1}\sigma^{-i}(\tau(d_i) ) N_i^{\sigma^{-1}}(\alpha) t^{-i}$$
which yields the assertion.

Indeed, a quick check of the defining identities yields that $\widetilde G_{\tau,\alpha,1}$ is indeed an anti-homomorphism:
$$\tau(d)\alpha t^{-1}=\widetilde G_{\tau,\alpha,1}(d)\widetilde G_{\tau,\alpha,1}(t)=\widetilde G_{\tau,\alpha,1}(td)=\widetilde G_{\tau,\alpha,1}(\sigma(d)t)=\sigma^{-1}(\tau(\sigma(d)))\alpha t^{-1}$$
 which holds as $\sigma$ and $\tau$ commute. We also require that
 $$\tau(a)=\widetilde G_{\tau,\alpha,1}(t^m)=\widetilde G_{\tau,\alpha,1}(t)^m=N_m^{\sigma^{-1}}(\alpha)a^{-1},$$
  which in turn holds because by assumption, we have $N_{m}^\sigma( \alpha)= \tau(a)a$:
 Observing that $$N_m^{\sigma^{-1}}(\alpha)= \sigma^{-(m-1)} (N_m^{\sigma}(\alpha))$$ and $\sigma(a)=a$, $\tau(a)=N_m^{\sigma^{-1}}(\alpha)a^{-1}$ is equivalent to
 $$\sigma^{-(m-1)} (\tau(a))=N_m^{\sigma^{-1}}(\alpha)\sigma^{-(m-1)} (a^{-1}),$$
  which is the same as the equation
$\tau(a)a=N_m^{\sigma^{-1}}(\alpha)a^{-1}$.
\\ $(ii)$ If $K[t;\sigma]/K[t;\sigma](t^m-a)$ is  an associative algebra, then $a\in F^\times$ and $m$ is a multiple of $n$. 
Suppose $\tau\in \Aut(K)_\sigma$ and $\alpha \in K^\times$. For $k=1$ the assertion is proved in $(i)$. For any $2\leq k < m$, the map
    $$
    G_{\tau,\alpha,k}: K[t;\sigma]/K[t;\sigma](t^m-a)\to K[t;\sigma]/K[t;\sigma](t^m-a^{-1}),
    $$
   $$
G ( \sum_{i=0}^{m-1}d_it^i ) = \sum_{i=0}^{m-1}\tau(d_i) (\alpha t^k)^i,
$$
 defined via $G|_{K}=\tau$ and $G(t)=\alpha t^k$,
    is a ring homomorphism if and only if  $k\equiv 1 \mod n$ and $\tau(a)=(N_{K/F}(\alpha))^{m/n}a^{-k}$.
    It is an isomorphism if and only if, in addition, we have $\gcd(k,m)=1$ \cite[Theorem 4.4]{NevPum2025}.

The associated
 $F^\tau$-linear anti-automorphism $\widetilde G_{\tilde\tau,\alpha,k} $ on $K[t;\sigma]/K[t;\sigma](t^m-a)$
 is defined by
 $$\widetilde G_{\tau,\alpha,k}(\sum_{i=0}^{m-1}d_it^i)= \sum_{i=0}^{m-1} (\alpha t^k)^{-i}\tau(d_i) =\sum_{i=0}^{m-1} N_i^{\sigma^{-k}}(\alpha) (t^{k})^{-i}\tau(d_i)$$
Therefore every $\tau$-semilinear (thus $F_0$-linear) anti-automorphism restricting to $\tau$ has the form $\widetilde G_{\tau,\alpha,k}$ for suitable $k$.
 \end{proof}

For associative cyclic algebras we thus obtain our first main result.

\begin{theorem}\label{thm:mainassociative cyclic1}
 Let  $(K/F,\sigma,a)$ be a cyclic algebra of degree $n$ and let $\tau \in {\rm Aut}(K)_\sigma$.
  Then there exists a $\tau$-semilinear (thus $F_0$-linear) anti-automorphism $\widetilde G_{\tau,\alpha,1}$ on the associative cyclic algebra $(K/F,\sigma,a)$ defined via
$$\widetilde G_{\tau,\alpha,1}(\sum_{i=0}^{n-1}d_it^i) =\sum_{i=0}^{n-1}(\alpha t)^{-i} \tau(d_i)= \sum_{i=0}^{n-1} N_i^{\sigma^{-1}}(\alpha) t^{-i} \tau(d_i),$$
 if and only if  there exists  some $\alpha\in K^\times$ such that
$$N_{K/F}( \alpha)= \tau(a)a.$$
 (ii) Every  $F_0$-linear anti-automorphism $\widetilde G_{\tau,\alpha,k}$  of an associative cyclic algebras which restricts  to some  $\tau \in {\rm Aut}(K)_\sigma$  has the form $\widetilde G_{\tau,\alpha,1}$.
\end{theorem}

\begin{proof}
$(i)$ follows from  Theorem \ref{thm:main0} $(i)$.
\\ $(ii)$ This follows from Theorem \ref{thm:main0} $(ii)$:
We first demonstrate that every  ring isomorphism $G_{\tau,\alpha,k}$  between $(K/F,\sigma,a)$ and $(K/F,\sigma,a^{-1})$ which restricts to some  $\tau \in {\rm Aut}(K)_\sigma$ has the form $ G_{\tau,\alpha,1}$. Because of the one-one correspondence between ring isomorphisms between an algebra and its opposite algebra, and anti-automorphisms of the algebra, this will prove our assertion.

Suppose $\tau\in \Aut(K)_\sigma$ and $\alpha \in K^\times$.  Then for any $2\leq k < n$, the map
    $$
    G_{\tau,\alpha,k}: (K/F,\sigma,a)\to (K/F,\sigma,a^{-1}),\quad
G ( \sum_{i=0}^{n-1}c_it^i ) = \sum_{i=0}^{n-1}\tau(c_i) (\alpha t^k)^i,
$$
    is a ring homomorphism if and only if  $k\equiv 1 \mod n$ and $\tau(a)=N_{K/F}(\alpha))a^{-k}$.
    It is an isomorphism if and only if, in addition, we have $\gcd(k,n)=1$ \cite[Theorem 4.4]{NevPum2025}.

 Since $k\equiv 1 \mod n$ cannot be satisfied for any $2\leq k < n$, there are no such isomorphisms, hence no associated
 $F^\tau$-linear anti-automorphism $\widetilde G_{\tilde\tau,\alpha,k} $ of
 $(K/F,\sigma,a)$ defined by
 $$\widetilde G_{\tau,\alpha,k}(\sum_{i=0}^{n-1}d_it^i)= \sum_{i=0}^{n-1} (\alpha t^k)^{-i}\tau(c_i) =\sum_{i=0}^{n-1} N_i^{\sigma^{-k}}(\alpha) (t^{k})^{-i}\tau(d_i)$$
 can exist unless $k=1$.

Therefore every $\tau$-semilinear (thus $F_0$-linear) anti-automorphism  has the form $\widetilde G_{\tau,\alpha,1}$.
\end{proof}

 Note that the anti-automorphisms $\widetilde G_{\tau,\alpha,1}$ is monomial of degree $n-1$, since $t^{-1}=a^{-1}t^{n-1}$.

In particular, if $F/F_0$ is a quadratic field extension generated by $\tau|_{F}$ (e.g., when $\widetilde G_{\tau,\alpha,1}$ is an involution of the second type), then for all $a\in F^\times$ we get $a\tau(a)=N_{F/F_0}(a)$. This means the necessary and sufficient condition
$$N_{K/F}( \alpha)= \tau(a)a$$
for the  $\tau$-semilinear automorphisms $\widetilde G_{\tau,\alpha,1}$ to exist, becomes
$$N_{K/F}( \alpha)= N_{F/F_0}(a)$$
for associative cyclic algebras $(K/F,\sigma,a)$,
forcing $N_{K/F}( \alpha)\in F_0^\times$, indeed even  $N_{K/F}( \alpha)\in N_{F/F_0}(F^\times)$.

\begin{corollary}
 Let $F$ be a number field  and $D$ and $(K/F,\sigma,a)$  be associative  division algebras of degree $n>2$.
  \\ (i) There does not exist any $F$-linear anti-automorphism on $D$. In particular, $D$ has no involution of the first kind.
   If $F$ has no non-trivial automorphisms (e.g., $F=\mathbb{Q}$), then $D$ has no anti-automorphisms of the second kind either.
  \\ (ii) If  there exists  some $\alpha\in K^\times$ such that $N_{K/F}( \alpha)= \tau(a)a$ for some suitable $id\not=\tau \in {\rm Aut}(K)_\sigma$ then $(K/F,\sigma,a)$ admits a $\tau$-semilinear anti-automorphism extending $\tau$, but no $F$-linear one.
  \end{corollary}

\begin{proof}
$(i)$ Since $D$ is a division algebra of degree greater than two over a number field, the exponent of $D$ is $n>2$ and so $D\not\cong D^{op}$ as $F$-algebras. Hence there do not exist any  $F$-linear anti-automorphisms of $D$. The rest of the assertion is clear.
\\ $(ii)$ is Theorem \ref{thm:main0} and $(i)$.
\end{proof}

The monomial anti-automorphisms of degree one are well behaved.

\begin{lemma}\label{le:1}
 Let $\tau, \tau',\tau_i \in {\rm Aut}(K)_\sigma$.
 \\
(i) For all positive integers $\ell$, we have $(\widetilde G_{\tau,\alpha,1})^\ell=id$  if and only if $\tau^\ell=id$ and
 $\tau^\ell(\alpha)\cdots \tau(\alpha)\alpha=1.$
 \\ (ii) $\widetilde G_{\tau,\alpha,1}$ is an involution if and only if $\tau^2=id$ and
 $ \tau(\alpha)\alpha=1.$
 \\ (iii)
 $\widetilde G_{\tau',\alpha',1}\circ \widetilde G_{\tau,\alpha,1}=\widetilde G_{\tau'\circ \tau,\tau'(\alpha)\alpha',1}$.
 \\ (iv) If the $\widetilde G_{\tau_i,\alpha_i}$  are  anti-automorphisms of $(K/F,\sigma,a)$, then for odd $s$,
 $$\widetilde G_{\tau_s,\alpha_s,1}\circ\cdots \circ  \widetilde G_{\tau_1,\alpha_1,1}=\widetilde G_{\tau_s\circ\cdots \circ \tau_1,\beta,1}$$
 with
 $$\beta=\tau_s(\dots \tau(\alpha_1)\dots ) \tau_{s-1}(\dots \tau(\alpha_2)\dots )\cdots \tau_1(\alpha_{s-1}) \alpha_s$$
 is an anti-automorphism of $(K/F,\sigma,a)$ which is $(\tau_s\circ\cdots \circ \tau_1)$-semilinear.
\end{lemma}

\begin{proof}
We will  again use the correspondence between isomorphism  and anti-automorphisms.
\\
$(i) $ We compute $( G_{\tau,\alpha,1})^\ell$ and obtain
$$( G_{\tau,\alpha,1})^\ell(\sum_{i=0}^{s-1}d_it^i)=
\sum_{i=0}^{s-1}\tau^\ell(d_i) \tau^\ell(N^{\sigma}_i(\alpha))\cdots \tau(N^{\sigma}_i(\alpha))N^{\sigma}_i(\alpha) t^{i}
=\sum_{i=0}^{s-1}\tau^\ell(d_i) (N^{\sigma}_i(\tau^\ell(\alpha))\cdots N^{\sigma}_i(\tau(\alpha))N^{\sigma}_i(\alpha) t^{i}
$$
$$=\sum_{i=0}^{s-1}\tau^\ell(d_i) N^{\sigma}_i(\tau^\ell(\alpha)\cdots \tau(\alpha)\alpha) t^{i}= (G_{\tau^\ell,\tau^\ell(\alpha)\cdots \tau(\alpha)\alpha}).
$$
Hence  $(G_{\tau,\alpha,1})^\ell=id$  if and only if
$$\tau^\ell(d_i) N^\sigma_i(\tau^\ell(\alpha)\cdots \tau(\alpha)\alpha)=d_i$$
for all $i\in \{0,\dots,n-1\}$  and all $d_i\in K$.

Looking at $i=0$ this forces $\tau^\ell=id$, so that the above condition is the same as having $\tau^\ell=id$ and
$$N^\sigma_i(\tau^\ell(\alpha)\cdots \tau(\alpha)\alpha)=1$$
for all $i\in \{1,\dots,n-1\}$.

For $i=1$ this yields $\tau^\ell(\alpha)\cdots \tau(\alpha)\alpha=1,$
which implies  $N^\sigma_i(\tau^\ell(\alpha)\cdots \tau(\alpha)\alpha)=1$ holds automatically
for all $i\in \{2,\dots,n-1\}$.
Thus if $G_{\tau,\alpha,1}$ has order $\ell$ then so does $\tau$ and $\tau^\ell(\alpha)\cdots \tau(\alpha)\alpha=1.$
 Conversely, these two conditions imply that $G_{\tau,\alpha,1}$ has order $\ell$.  Moreover, $\widetilde G_{\tau,\alpha,1}$ has order $\ell$ if and only so does $\widetilde G_{\tau,\alpha,1}$.
\\ $(ii)$ follows from $(i)$, $(iii)$ is trivial, and
 $(iv)$ follows from $(iii)$.
 \end{proof}

 We recover a well-known norm criterium for the existence of involutions on cyclic algebras.

\begin{corollary}
Let  $(K/F,\sigma,a)$ be an associative cyclic algebra of degree $n$. Then there exists an  anti-automorphism
$\widetilde G_{\alpha,\tau,1}$ of the first kind on $(K/F,\sigma,a)$, if and only if $\tau\in\Gal(K/F)$  and there exists  some $\alpha\in K^\times$ such that
$$N_{K/F}( \alpha)= a^2.$$
\end{corollary}

\begin{corollary}\label{cor:inv}
Let  $(K/F,\sigma,a)$ be an associative cyclic algebra of degree $n$ and let $\tau \in {\rm Aut}(K)_\sigma$. Let $F_0=F^\tau$ be the fixed field of $\tau$.
Then there exists an $F_0$-linear involution $\widetilde G_{\tau,\alpha,1}$ on $(K/F,\sigma,a)$
 if and only if $\tau^2=id$ and if there exists  some $\alpha\in K^\times$ such that
$$N_{K/F}( \alpha)=  a\tau(a) \text{ and }  \tau(\alpha)\alpha=1.$$
\end{corollary}

\begin{proof}
By Theorem \ref{thm:mainassociative cyclic1}, such an involution exists if and only if there exists  some $\alpha\in K^\times$ such that
$N_{K/F}( \alpha)= a\tau(a).$ By Lemma \ref{le:1}, $\widetilde G_{\tau,\alpha,1}$ is an involution if and only if $\tau^2=id$ and
 $ \tau(\alpha)\alpha=1.$
\end{proof}

\section{Monomial anti-automorphisms on associative cyclic algebras from monomial isomorphisms of degree  greater than one}

Monomial $\tau$-semilinear anti-automorphisms of degree $\ell>1$ greater than one do not exist on cyclic algebras when
$\tau\in \Aut(K)_\sigma$ (Theorem \ref{thm:mainassociative cyclic1} $(ii)$).

We thus turn to the general case when $\tau\in \Aut(K)$ does not commute with $\sigma$. By Lemma \ref{le:degreek}, then there exists an integer $k\in \{2,\dots, n-1 \}$ such that $\gcd (k, n)=1$ and
$$\tau \sigma\tau^{-1}=\sigma^k.$$

 \subsection{Monomial ring isomorphisms of degree $k>1$ between cyclic algebras}\label{sec:isocyclic}

  When $\tau\in \Aut(K)$ such that $\tau \sigma\tau^{-1}=\sigma^k$ for some positive integer $k$,  $\gcd (k, n)=1$, we will now show that there exist ring isomorphisms between associative cyclic algebras that are  \emph{monomial of degree $k$} with $k>2$, i.e. where $t$ gets mapped to $\alpha t^k$ for some $k>1$ and $\alpha\in K^\times$.

 \begin{lemma}\label{L:Ktsigmatoitself}\cite[Lemma 3.1]{NevPum2025}
    Let $\tau\in \Aut(K)$, $k$ a positive integer,  and $\alpha\in K^\times$.  Then the map generated by $G(c)=\tau(c)$ for all $c\in K$ and $G(t)=\alpha t^k$ defines a ring automorphism  on $K[t;\sigma]$ if and only if $\tau\sigma\tau^{-1}=\sigma^k$.
\end{lemma}

 These ring automorphisms on $K[t;\sigma]$  canonically induce monomial ring isomorphisms of degree $k$ between cyclic algebras in very special cases, and only when the cyclic algebras are associative. This is a consequence of the next results. Note that Theorem \ref{general_isomorphism_theorem2} generalizes \cite[Theorem 4.4]{NevPum2025} and that its statement includes nonassociative Petit rings as well.

 \begin{theorem}  \label{general_isomorphism_theorem2}
    Suppose that $\tau \sigma\tau^{-1}=\sigma^k$ for some $k\in \{2,\dots, n-1 \}$ with $\gcd (k, n)=1$.
    Then the map
 $$
  G_{\tau,\alpha,k}:K[t;\sigma]/K[t;\sigma](t^m-a)\to K[t;\sigma]/K[t;\sigma](t^m-b),
  $$
  $$
G_{\tau,\alpha,k} ( \sum_{i=0}^{n-1}d_it^i ) = \sum_{i=0}^{n-1}\tau(d_i) (\alpha t^k)^i  = \sum_{i=0}^{n-1} \tau(d_i) N_i^{\sigma^k}(\alpha) t^{ik}
$$
 between the Petit rings $\mathbb{S}_{t^m-a}$ and $\mathbb{S}_{t^m-b}$,
is a ring isomorphism if and only  the following conditions are satisfied.
\\ (i)  $\gcd(k,m)=1$.
\\ (ii)  $n\mid m$.
\\ (iii) $a, b\in F^\times$.
\\ (iv)
$(N_{K/F}(\alpha))^{m/n} b^k=\tau(a)$.
\end{theorem}

Note that conditions $(ii)$ and $(iii)$ imply that $\mathbb{S}_{t^m-a}$ and $\mathbb{S}_{t^m-b}$ are associative, thus classical quotient rings.

\begin{proof}
 Suppose first that the map $G=G_{\tau,\alpha,k}$ defined by the rule
\begin{equation}\label{E:formulaG(t)}
G( \sum_{i=0}^{m-1}d_it^i ) = \sum_{i=0}^{m-1}\tau(d_i) (\alpha t^k)^i
\end{equation}
is a well-defined isomorphism.

Since $G$ is an isomorphism, there must exist an inverse homomorphism $H$.
Evidently $H(d)=\tau^{-1}(d)$ for all $d\in K$.  Since $G$ is monomial, so is its inverse, and thus we may assume $H(t)=\beta t^\ell$ for some $\beta \in K^\times$ and $1\leq \ell < m$.  Since $H(G(t))=t$ we must have $\ell k \equiv 1 \mod m$ and in particular $\gcd(k,m)=1$ which is $(i)$.

The identity $G(td)=G(\sigma(d)t)$ implies $\alpha t^k \tau(d) = \tau(\sigma(d)) \alpha t^k$ for all $d\in K$, whence we must have
$$ \sigma^k(\tau(d)) = \tau(\sigma(d))$$ for all $d\in K$ which we indeed have by our assumption.

For $G$ to be well-defined, we require for all $0<s <m$ that $G(t^s)=G(t)^s$ is well-defined.  Since $k>1$, the minimal value of $r$ such that $rk\geq m$ satisfies $0<r<m$ and thus by \cite[Theorem 3.3]{NevPum2025}, we conclude that $\alpha b = \sigma^{m}(\alpha)\sigma^k(b)$.
Conversely, \cite[Theorem 3.3]{NevPum2025} further assures us that this hypothesis implies all powers of $\alpha t^k$ are well-defined so $G$ gives indeed a well-defined map.

 By  \cite[Theorem 3.3]{NevPum2025} we know  that for any $0<s<m$, the element $G(t^s)=G(t)^s \in K[t;\sigma]/K[t;\sigma](t^m-b)$ will be a monomial of degree $[sk]_{m}$, the residue of $sk$ mod $m$. In order for $G$ to further preserve the relation $t^sd=\sigma^s(d)t^s$ for all $d\in K$, we require $G(t^s)\tau(d) = \tau(\sigma^s(d))G(t^s)$, which will hold if and only if
$$
\sigma^{[sk]_{m}}(\tau(d))=\tau(\sigma^{s}(d))
$$
for all $d\in K$. On the other hand, by assumption $ \sigma^k = \tau\sigma\tau^{-1}$, so
 by induction
$ \tau\sigma^s \tau^{-1}=\sigma^{ks}$. Now $\sigma$ has order $n$ so this is the same as
$$
\sigma^{[sk]_{n}}(\tau(a))=\tau(\sigma^{s}(a))
$$
for all $0<s<m$. However, $\sigma^{[sk]_{n}}=
\sigma^{[sk]_{m}}$ if and only if $n\mid m$, so we obtain $(ii)$. In return, $n\mid m$ implies $G(t^s)\tau(d) = \tau(\sigma^s(d))G(t^s)$ for all $0<s<m$ and $d\in K$.

Since $n\mid m$, the relation $\alpha b = \sigma^{m}(\alpha)\sigma^k(b)$ simplifies to $\sigma(b)=b$, or $b\in F$, yielding part of $(iii)$.

 Furthermore, $G$ must satisfy $G(t^{m}-a)=0$.  By the above, we may evaluate $G(t^{m})=G(t)^{m}$ as follows.
By \cite[Theorem 3.3]{NevPum2025}, for every positive integer $s$, we know that
\begin{equation}\label{G(t)^s}
(\alpha t^k)^s = N_{s}^{\sigma^k}(\alpha)\prod_{i=1}^{\lfloor \frac{sk}{{m}}\rfloor} \sigma^{sk-i{m}}(b)t^{[sk]_{m}}
\end{equation}
where $\lfloor \frac{sk}{m}\rfloor=j$ when $j{m}\leq sk<(j+1){m}$ and $[sk]_{m} = sk-\lfloor \frac{sk}{m}\rfloor m$ denotes the residue of $sk$ mod $m$ in the interval $[0,m-1]$. With  $\sigma$ of order $n$ and $n\mid m$ this yields
$$
G(t)^{m} = N_{m}^{\sigma^k}(\alpha) \prod_{i=1}^{\lfloor \frac{mk}{{m}}\rfloor} \sigma^{mk-i{m}}(b)t^{[mk]_{m}}=N_{m}^{\sigma^k}(\alpha) \prod_{i=1}^{k}\sigma^{m(k-i)}(b)=N_{m}^{\sigma^k}(\alpha) b^k.
$$
Since $G(t^{m})=G(a)$, this is equivalent to the condition
$$
N_{m}^{\sigma^k}(\alpha) b^k=\tau(a).
$$
Also, since $\gcd(k,n)=1$ we have $\langle \sigma\rangle=\langle \sigma^k\rangle $ for the Galois group of $K/F$, so that $N_{n}^{\sigma^k}=N_{n}^{\sigma}$, thus
$$
N_{m}^{\sigma^k}(\alpha) = (N_{n}^{\sigma^k}(\alpha))^{m/n} = (N^\sigma_{n}(\alpha))^{m/n}  = (N_{K/F}(\alpha))^{m/n}.
$$
In particular, $(N_{K/F}(\alpha))^{m/n}(\alpha) b^k=\tau(a)$ implies that $\tau(a)\in F$, hence $a\in F$, which is the rest of condition $(iii)$.

Finally, if $k,n,\alpha,a,b,\tau$,   satisfy conditions $(i)$ to $(iv)$, and $k\in \{2,\dots, m-1 \}$ such that $\gcd (k, n)=1$ and $\sigma^k\tau = \tau\sigma$, then the $F^\tau$-linear map $G_{\tau,\alpha,k}:K[t;\sigma]/K[t;\sigma](t^n-a)\to K[t;\sigma]/K[t;\sigma](t^n-b)$ is a well-defined homomorphism of $F^\tau$-modules, and respects the relations $G(t^{n})=G(a)$ and $G(t^su)=G(\sigma^s(u)t^s)$ for all $0<s<n$ and all $u\in K$.  Therefore in particular it satisfies $G(p)G(q)=G(pq)$ and $G(p+q)=G(p)+G(q)$ for all $p,q\in K[t;\sigma]/K[t;\sigma](t^n-a)$, which implies it is an homomorphism of $F^\tau$-algebras.  The final condition ensures it is an isomorphism.
\end{proof}

Thus if $\tau \sigma\tau^{-1}=\sigma^k$ for some $k\in \{2,\dots, n-1 \}$ with $\gcd (k, n)=1$ and  $n\nmid m$, e.g. when $n> m$,
   then there is no ring isomorphism of the kind
    $G_{\tau,\alpha,k}:K[t;\sigma]/K[t;\sigma](t^m-a)\to K[t;\sigma]/K[t;\sigma](t^m-b)$ extending $\tau$. Indeed, we can show more.

 \begin{proposition}\label{prop:newconditiononk}
   Suppose that $\tau \sigma\tau^{-1}=\sigma^k$ for some $k\in \{2,\dots, n-1 \}$ with $\gcd (k, n)=1$ and suppose that $n> m$.
   Then there is no ring isomorphism
    $G:K[t;\sigma]/K[t;\sigma](t^m-a)\to K[t;\sigma]/K[t;\sigma](t^m-b)$ extending $\tau$.
 \end{proposition}

 \begin{proof}
 Suppose  that $G(t) = \sum_{i=0}^{m-1} \alpha_i t^i$ for some $\alpha_i \in K$.
Then we have
\begin{equation} \label{eqn:automorphism_necessity_theorem 1}
G(tz) = G(t)G(z) = (\sum_{i=0}^{m-1} \alpha_i t^i ) \tau(z) =
\sum_{i=0}^{m-1} \alpha_i \sigma^{i}(\tau(z)) t^i \end{equation} and
\begin{equation} \label{eqn:automorphism_necessity_theorem 2}
G(tz) =
G(\sigma(z)t) = \sum_{i=0}^{m-1} \tau(\sigma(z)) \alpha_i t^i
\end{equation}
for all $z \in K$. Comparing the coefficients of $t^i$
in \eqref{eqn:automorphism_necessity_theorem 1} and
\eqref{eqn:automorphism_necessity_theorem 2} we obtain
$$\alpha_i \sigma^{i}(\tau(z)) = \tau(\sigma(z)) \alpha_i$$
for all $i \in\{ 0, \ldots, n-1\}$ and all $z \in K$. This is equivalent to
$$\alpha_i (\sigma^i(\tau(z)) - \tau\sigma(z)  )=0,$$
for all $i \in\{ 0, \ldots, m-1 \}$ and all $z \in K$.
Since $k\in \{2,\dots, n-1 \}$, $\gcd (k, n)=1$,
 $\tau \sigma\tau^{-1}=\sigma^k$
 and $\sigma$ has order $n > m$, we know that
 $$\tau \sigma\tau^{-1}\not=\sigma^i$$
 for all  $i \in\{ 0, \ldots, m-1\}$ with $i\not=k$ so that  $\alpha_i = 0$ for
all $ i \in\{ 0, \ldots, m-1\}$  with $i\not=k$.

 If  $k\geq m$ then this implies that $G(t)=0$ and so $G$ cannot be an isomorphism. This implies we need $k<m$.
If $k<m$ then this implies that $G(t) = \alpha t^k$ for some $\alpha\in K^{\times}$.
We obtain $G=G_{\tau,\alpha,k}$
since $G|_{K}=\tau$ and $G(t)=\alpha t^k$.
However, the map $G=G_{\tau,\alpha,k}$ is only an isomorphism when $n\mid m$, so this cannot happen.
\end{proof}

 Moreover, Theorem \ref{general_isomorphism_theorem2} showed us that monomial isomorphism $G_{\tau,\alpha,k}$ of degree $k>1$ that extend $\tau$ can only exist between \emph{associative} cyclic algebras $(K/F,\sigma, a)$ and $(K/F,\sigma, b)$. Note also that when we work with associative algebras that have the same maximal subfield $K/F$ that we can assume w.o.l.o.g that their definitions employ the same generator of the field extension $K/F$.

 \begin{corollary}  \label{cor:general_isomorphism_theorem2}
 Suppose that $\tau\in \Aut(K)$ and $\tau \sigma\tau^{-1}=\sigma^k$ for some $k\in \{2,\dots, n-1 \}$ such that $\gcd (k, n)=1$.
 Let $\alpha\in K^\times$.
 There exists a ring isomorphism $$G_{\tau,\alpha,k}: (K/F,\sigma,a)\to (K/F,\sigma,b)$$
  if and only if
 $(K/F,\sigma,a)$ is associative and
 $$N_{K/F}(\alpha) b^k=\tau(a).$$
 \end{corollary}

 This is Theorem \ref{general_isomorphism_theorem2} for $n=m$.

 \begin{remark}
 Note that if $K[t;\sigma]/K[t;\sigma](t^n-a)$ is not
associative, its left nucleus is $K$. Since any isomorphism  $G:K[t;\sigma]/K[t;\sigma](t^n-a)\to K[t;\sigma]/K[t;\sigma](t^n-b)$ preserves the left nucleus, therefore $G(K) = K$ is automatically satisfied  and so $G \vert_K = \tau$ for some
 $\tau \in \text{Aut}_F(K)$ will always hold. When the algebra is associative, this need not be the case.
\end{remark}

\subsection{Monomial anti-automorphisms  of degree $k<n-1$ on cyclic algebras}

As an immediate consequence of the results in Section \ref{sec:isocyclic} we obtain  criteria on the existence of anti-automorphisms on associative quotient algebras which extend some $\tau\in \Aut(K)$.

\begin{theorem}  \label{general_antiautomorphism_theorem2}
Let $K[t;\sigma]/K[t;\sigma](t^m-a)$ be an associative quotient algebra. Suppose that $\tau \sigma\tau^{-1}=\sigma^k$ for some $k\in \{2,\dots, n-1 \}$ with $\gcd (k, n)=1$. Then the map
 $$
  \widetilde G_{\tau,\alpha,k}:K[t;\sigma]/K[t;\sigma](t^m-a)\to K[t;\sigma]/K[t;\sigma](t^m-a)
  $$
  $$
 \widetilde G_{\tau,\alpha,k} ( \sum_{i=0}^{m-1}d_it^i ) = \sum_{i=0}^{m-1} (\alpha t^k)^{-i} \tau(d_i) = \sum_{i=0}^{m-1}
  N_i^{\sigma^{-k}}(\alpha) t^{-ik} \tau(d_i)
$$
is a $\tau$-semilinear anti-automorphism if and only  the following two conditions are satisfied.
\\ (i)  $\gcd(k,m)=1$.
\\ (ii)
$(N_{K/F}(\alpha))^{m/n} =\tau(a)a^k$.
\end{theorem}

This follows from Theorem \ref{general_isomorphism_theorem2}, since $K[t;\sigma]/K[t;\sigma](t^m-a)$ is associative exactly when $n\mid m$ and $a\in F^\times$.

\begin{theorem}  \label{general_antiautomorphism_theorem4}
Let  $(K/F,\sigma,a)$ be an associative cyclic algebra. Suppose that  $\tau\in \Aut(K)$ and $\tau \sigma\tau^{-1}=\sigma^k$ for some $k\in \{2,\dots, n-1 \}$ such that $\gcd (k, n)=1$.
 Let $\alpha\in K^\times$.
 There exists a $\tau$-semilinear anti-automorphism
 $\widetilde G_{\tau,\alpha,k}$ on $ (K/F,\sigma,a)$, if and only if
  $$N_{K/F}(\alpha)=\tau(a)a^{k}.$$
\end{theorem}

This follows from Corollary \ref{cor:general_isomorphism_theorem2}.

An anti-automorphism $f$ on an associative cyclic algebra $(K/F,\sigma,a)$  does not need to map $K$ to $K$ as we have always assumed in the previous calculations. 
We only know that $f(K)=K'$ hence is a field extension isomorphic to $K$. 

For any two anti-automorphisms $f,g$ of $(K/F,\sigma,a)$ there exists $u\in  (K/F,\sigma,a)$ such that $fg=i_u$ meaning $g = f^{-1} i_u$.
Thus an arbitrary anti-automorphism $g$ of an associative cyclic algebra $(K/F,\sigma,a)$ can be decopmosed into an inner autmorphism followed by the inverse of some 

 Therefore, while an arbitrary anti-automorphism $g$ does not generally map $K$ to itself, it can always be decomposed into an inner automorphism followed by the inverse of some anti-automorphism $\widetilde G_{\tilde\tau,\alpha,k}$.

Alternatively, the computations up to now, both for isomorphisms and anti-automorphisms, hold analogously when we only assume that $\tau:K\to K'$ is a field automorphism and use $\sigma$ and $\sigma'$ as generators for the respective Galois groups of $K/F$ and $K'/F$.

\section{Examples of anti-automorphisms of the second kind  on algebras without involutions}\label{sec:Pat}

In this section we study two of the proofs given in  \cite{MST} and in the process also look at anti-automorphisms of the twisted Laurent series ring $K((t;\sigma))$. The approach presented in  \cite[Proposition 3.4]{MST} is slightly different to ours and can be summarized in the following proposition.

\begin{proposition}\label{prop:PatGen} 
Suppose that $\sigma\tau\sigma=\tau$.
 Then the associative cyclic algebra $(K/F,\sigma,a)$ admits a $\tau$-semilinear anti-automorphism $\widetilde H_{\tau,\alpha,1}:(K/F,\sigma,a) \to (K/F,\sigma,a)$  of the second kind,
 $$\widetilde H_{\tau,\alpha,1}(\sum_{i=0}^{n-1}d_it^i)=\sum_{i=0}^{n-1}   (\alpha t)^i \tau(d_i)=\sum_{i=0}^{n-1} \sigma^i(\tau(d_i) ) N_i^\sigma(\alpha) t^i,$$
 which restricts to $\tau$ and maps $t$ to $\alpha t$ for some $\alpha\in K^\times$, if and only if
 $$\tau(a)=N_{K/F}(\alpha) a.$$
   \end{proposition}

 \begin{proof}
 If $\widetilde H_{\tau,\alpha,1}$ is an anti-automorphism then $\widetilde H_{\tau,\alpha,1}(t^n)=\widetilde H_{\tau,\alpha,1}(a)$ and $\widetilde H_{\tau,\alpha,1}(t)^n= (\alpha t)^n=N_{K/F}(\sigma)t^n$, so that $\tau(a)=N_{K/F}(\alpha) a.$ Moreover, for all $d\in K$ we get
 $\widetilde H_{\tau,\alpha,1}(td)=\widetilde H_{\tau,\alpha,1}(\sigma(d)t)=\sigma(\tau(\sigma(d)) \alpha t$
 and $\widetilde H_{\tau,\alpha,1}(td)=\widetilde H_{\tau,\alpha,1}(d)\widetilde H_{\tau,\alpha,1}(t)=\tau(d) \alpha t$, so that
 $\sigma\tau\sigma=\tau$ which is true by assumption. If these conditions hold then $\widetilde H_{\tau,\alpha,1}$ is an anti-homomorphism. Since $\alpha\in K^\times$ it is bijective. The rest of the assertion is clear then.
 \end{proof}
 
 We see that the  anti-automorphisms $\widetilde H_{\tau,\alpha,1}$ can be compared with the anti-automorphisms $\widetilde G_{\tau,\beta,n-1}$. Both maps are monomial of degree one, since $(t^{n-1})^{-1}=a^{-1}t$.
  Under the right assumptions on $\tau$, it is possible to analogously define anti-automorphisms $\widetilde H_{\tau,\alpha,k}$ mapping $t$ to  $\alpha t^k$, and compare these with the anti-automorphisms $\widetilde G_{\tau,\beta,s}$ where $k=n-s$, since $(t^{s})^{-1}=a^{-1}t^{n-s}$.
  
  Both approaches are equally valid and demonstrate the circular nature of the anti-automorphisms we define. We chose our approach as we already had results on isomorphisms between the algebras and their opposite algebras which we could employ.

\subsection{Anti-automorphisms  on cyclic algebras without involutions}

In this section, let $m>0$, $n>2$, and let $K/F$ be a cyclic Galois extension of degree $n$ with Galois group $ \Gal(K/F)=\langle \sigma\rangle$. Let $G_0$ be the group generated by $\sigma$ and $\tau\in \Aut(K)$ with $$\sigma^n=id, \quad\tau^{4m}=id, \quad \sigma\tau\sigma=\tau.$$
 Let $K/\mathbb{Q}$ be a  Galois extension of degree $mn$ with Galois group $G_0$.
  Let $(K/F,\sigma,a)$ be an associative cyclic algebra of degree $n$.

 \begin{proposition}\label{prop:Pat0} 
 The associative cyclic algebra $(K/F,\sigma,a)$ admits a $\tau$-semilinear anti-automorphism $\widetilde H_{\tau,\alpha,1}:(K/F,\sigma,a) \to (K/F,\sigma,a)$  of the second kind,
 $$\widetilde H_{\tau,\alpha,1}(\sum_{i=0}^{n-1}d_it^i)=\sum_{i=0}^{n-1}   (\alpha t)^i \tau(d_i)=\sum_{i=0}^{n-1} \sigma^i(\tau(d_i) ) N_i^\sigma(\alpha) t^i,$$
 which restricts to $\tau$ and maps $t$ to $\alpha t$ for some $\alpha\in K^\times$, if and only if
 $$\tau(a)=N_{K/F}(\alpha) a.$$
 If, in particular, $a\in \mathbb{Q}^\times$, then $\widetilde H_{\tau,\alpha,1}$ is an anti-automorphism if and only if $N_{K/F}(\alpha)=1$.
If $(K/F,\sigma,a)$ is a division algebra, then $(K/F,\sigma,a)$ does not admit an involution.
   \end{proposition}

This is a slight generalization of \cite[Proposition 3.4]{MST}.

 \begin{proof}
It remains to show the second assertion, since this follows from Proposition \ref{prop:PatGen}.
  If $D=(K/F,\sigma,a)$ is a division algebra, then $D\not\cong D^{op}$ as $F$-algebras, since its degree is $n>2$ hence so is its exponent, as we work over number fields here. Therefore $D$ does not admit an involution, as already observed in the proof of \cite[Proposition 3.4]{MST}.
 \end{proof}

   In  the proof of \cite[Proposition 3.4]{MST} (and under the additional assumption that  $a\in \mathbb{Q}^\times$),  the anti-automorphism $\widetilde H_{\tau,1,1}:(K/F,\sigma,a) \to (K/F,\sigma,a),$
 $$\widetilde H_{\tau,1,1}(\sum_{i=0}^{n-1}d_it^i)=\sum_{i=0}^{n-1}  t^i \tau(d_i)=\sum_{i=0}^{n-1} \sigma^i(\tau(d_i) ) t^i,$$
 is used where $\tau(a)=a$ by assumption.

  Let us use this setup to see what happens when we apply our approach. Let $\alpha\in K^\times$ as usual.
  We have $\sigma\tau\sigma=\tau$ by assumption, which is the same as $\tau\sigma\tau^{-1}=\sigma^{-1}$ which in turn is equivalent to $\tau\sigma\tau^{-1}=\sigma^{n-1}$, so we can apply Theorem  \ref{general_antiautomorphism_theorem4} for $k=n-1$. This tells us that
   if
   $$N_{K/F}(\alpha)=\tau(a)a^{n-1},$$
  then there also exists a $\tau$-semilinear monomial anti-automorphism
 $\widetilde G_{\tau,\alpha,n-1}$ of degree one on $ (K/F,\sigma,a)$, defined via
 $$
\widetilde G_{\tau,\alpha,n-1} ( \sum_{i=0}^{n-1}d_it^i ) = \sum_{i=0}^{n-1} (\alpha t^{n-1})^{-i} \tau(d_i).
$$

 Thus monomial anti-automorphisms $\widetilde G$ such that $\widetilde G|_K=\tau$ and that  map $t$ to $(\alpha t^{n-1})^{-1}$,  exist for the right choice of $a\in F^\times$ as well and we obtain a second explicit construction of an anti-automorphism of the second kind on $(K/F,\sigma,a)$ that is monomial of degree one.

 \begin{proposition}\label{prop:Pat}
 The associative cyclic algebra $(K/F,\sigma,a)$ admits an anti-automorphism of the second kind $\widetilde G_{\tau,\alpha,n-1}$ that restricts to $\tau$ and maps $t$ to $(\alpha t^{n-1})^{-1}$ for some $\alpha\in K^\times$, if and only if
 $$N_{K/F}(\alpha)=\tau(a)a^{n-1}.$$
 If, in particular, $a\in \mathbb{Q}^\times$, then $\widetilde G_{\tau,\alpha,n-1}$ is an anti-automorphism if and only if $N_{K/F}(\alpha)=a^n$ (e.g., choose $\alpha=a$).
If $D$ is a division algebra, then $D$ does not admit an involution.
   \end{proposition}

\begin{remark}
Put for example $\alpha=1$. Then there exists  a $\tau$-semilinear anti-automorphism $\widetilde G_{\tau,1,n-1}: (K/F,\sigma,a)\to (K/F,\sigma,a)$ if and only $\tau(a)a^{n-1}=1.$
   Suppose that additionally  $a\in \mathbb{Q}^\times$, then  there exists  a $\tau$-semilinear anti-automorphism
 $\widetilde G_{\tau,1,n-1}: (K/F,\sigma,a)\to (K/F,\sigma,a)$
  if and only $a^{n}=1.$
   But for $a=1$ which is the only option here, we  obtain the trivial case that $(K/F,\sigma,1)$ is split. Note that we run into the same problem when we choose any $a\in F^\times$ such that $\tau(a)=a$, since also in that case  there exists  a $\tau$-semilinear anti-automorphism
 $\widetilde G_{\tau,1,n-1}$ if and only if $a^n=1$, forcing $a=\zeta_n$ to be an $n$th root of unity, hence we end up again in the trivial case that $(K/F,\sigma,\zeta_n)$ is split.
\end{remark}

\section{Anti-automorphisms on the ring  twisted Laurent series} \label{sec:Laurent}

Let $K/F$ be a cyclic Galois field extension of degree $n$ with Galois group generated by $\sigma$, and let $K((t;\sigma))$ be the division ring of twisted Laurent series in the indeterminate $t$ with the usual addition and multiplication given by $ta=\sigma(a)t$.

Then $K((t;\sigma))=(K((t^n))/F((t^n)),\sigma, t^n)$ is a cyclic division algebra  of degree $n$ over its center $F((t^n))$, where we write $\sigma$ also for the unique extension of $\sigma$ to $K((t^n))$ that leaves $t$ invariant.

By \cite[Proposition 4.3]{MST}, every anti-automorphism $f$ on $K((t;\sigma))$ satisfies  $\bar f \sigma (\bar f)^{-1}=\sigma^{-1}$, where $\bar f$ denotes the canonical automorphism on the residue field $K$ that is induced by $f$.
This obviously implies that for every anti-automorphism $f$ on $K((t;\sigma))$ that restricts to some $\tau\in\Aut(K)$, we have  $\tau \sigma \tau^{-1}=\sigma^{-1}$, equivalently $\tau \sigma \tau^{-1}=\sigma^{n-1}$.

Suppose that  $\tau\in \Aut(K)$. Then either $\tau$ and $\sigma$ permute or there exists some $k\in \{2,\dots, n-1 \}$ such that $\gcd (k, n)=1$  and $\tau \sigma\tau^{-1}=\sigma^k$. Every automorphism $\tau$ extends canonically to some unique $\tau\in \Aut(K((t^n)))$ that maps $t^n$ to $t^n$.
 Let $\alpha\in K((t^n))^\times$ and put $x=t^n$.  An element in $(K((t^n))/F((t^n)),\sigma,t^n)$ has the form
 $$
 \sum_{i=0}^{n-1}(\sum_{j_i}c_{i,{j_i}} x^{j_i}) t^i$$
 with the $c_{i,{j_i}}\in K$. By  Corollary \ref{cor:general_isomorphism_theorem2} and Theorem   \ref{general_antiautomorphism_theorem4},
 there exists a $\tau$-semilinear anti-automorphism $\widetilde G_{\tau,\alpha,k}$ on the cyclic algebra $ (K((t^n))/F((t^n),\sigma, t^n)$
  if and only if
 $$N_{K((t^n))/F((t^n))}(\alpha)=\tau(t^n)t^{nk}=t^{n(k+1)},$$
 and this $\tau$-semilinear anti-automorphism $\widetilde G_{\tau,\alpha,k}$ is canonically defined via the isomorphism
 $$G_{\tau,\alpha,k}: (K((t^n))/F((t^n)),\sigma,t^n)\to (K((t^n))/F((x)),\sigma, (t^n)^{-1}),$$
  $$ G_{\tau,\alpha,k} ( \sum_{i=0}^{n-1}(\sum_{j_i}c_{i,{j_i}} x^j) t^i ) =
\sum_{i=0}^{n-1}(\sum_{j_i}\tau(c_{i,{j_i}}) x^{j_i}) (\alpha t^k)^i.$$

 We compute explicitly that
  $$
 G_{\tau,\alpha,k} ( \sum_{i=0}^{n-1}(\sum_{j_i}c_{i,{j_i}} x^{j_i}) t^i ) =
\sum_{i=0}^{n-1}(\sum_{j_i}\tau(c_{i,j}) x^{j_i}) (\alpha t^k)^i =
\sum_{i=0}^{n-1}(\sum_{j_i}\tau(c_{i,{j_i}}) x^{j_i})  N_i^{\sigma^k}(\alpha) t^{ik}
$$
$$ = \sum_{i=0}^{n-1}(\sum_{j_i}\tau(c_{i,{j_i}}) t^{n{j_i}}  N_i^{\sigma^k}(\alpha) t^{ik}
=  \sum_{i=0}^{n-1}\sum_{j_i}\tau(c_{i,{j_i}}) \sigma^{n{j_i}}( N_i^{\sigma^k}(\alpha)) t^{n{j_i}}  t^{ik}
=  \sum_{i=0}^{n-1}\sum_{j_i}\tau(c_{i,{j_i}})  N_i^{\sigma^k}(\alpha) t^{n{j_i}+ik}.$$
The last equality holds because $\sigma$ has order $n$.

 By \cite[Proposition 4.3]{MST}, however, as observed above, the existence of a corresponding anti-automorphism $\widetilde G_{\tau,\alpha,k}$ forces $k=n-1$, other exponents $k$ cannot exist; and so
   there exists a $\tau$-semilinear anti-automorphism $\widetilde G_{\tau,\alpha,n-1}$ on $(K((t^n))/F((t^n)),\sigma,t^n)$
  if and only if
  $$N_{K((t^n))/F((t^n))}(\alpha)=\tau(t^n)t^{n(n-1)}=t^{n^2}.$$

  Additionally, if $\sigma$ and $\tau$ commute, there exists a $\tau$-semilinear anti-automorphism $\widetilde G_{\tau,\alpha,1}$ on the cyclic algebra $(K((t^n))/F((t^n)),\sigma,t^n)$,
  if and only if
 $$N_{K((x))/F((x))}(\alpha)=\tau(t^n)t^n=t^{2n}.$$
 The anti-automorphism $\widetilde G_{\tau,\alpha,1}$ is again canonically given by the isomorphism
 $$ G_{\tau,\alpha,1}: (K((t^n))/F((t^n)),\sigma,t^n)\to (K((t^n))/F((t^n)),\sigma, t^n),$$
  $$
 G_{\tau,\alpha,1}( \sum_{i=0}^{n-1}(\sum_{j_i}c_{i,{j_i}} x^{j_i}) t^i ) =
\sum_{i=0}^{n-1}(\sum_{j_i}\tau(c_{i,{j_i}}) x^{j_i}) (\alpha t)^i.
$$

   \begin{theorem} \label{thm:Patrelvant}
   Let $K/F$ be a cyclic Galois field extension of degree $n$ with Galois group generated by $\sigma$. Suppose that $\alpha\in K((t^n))^\times$.
   The map $\widetilde G_{\tau,\alpha,k}:K((t;\sigma))\to K((t;\sigma))$,
    $$
\widetilde  G_{\tau,\alpha,k}(\sum_{i=0}^{n-1}(\sum_{j_i}c_{i,{j_i}} x^{j_i}) t^i)  =  \sum_{i=0}^{n-1}\sum_{j_i} N_i^{\sigma^{-k}}(\alpha) t^{-n{j_i}-ik}\tau(c_{i,{j_i}})
$$
  is a $\tau$-semilinear anti-automorphism if and only if $k=n-1$, $\tau \sigma\tau^{-1}=\sigma^{n-1}$, and
  $$N_{K((t^n))/F((t^n))}(\alpha)=t^{n^2}.$$
  
 In particular, if $\sigma$ and $\tau$ commute then $\widetilde G_{\tau,\alpha,k}$ is an anti-automorphism if and only if $n=2$ and $k=1$. Conversely, if $\widetilde G_{\tau,\alpha,1}$ is an anti-automorphism, then $\sigma$ and $\tau$ must commute.
  \end{theorem}

\begin{proof}
 For every $\tau\in \Aut(K)$ there exists some
  $k\in \{1,\dots, n-1 \}$ with $\gcd (k, n)=1$ such that $\tau \sigma\tau^{-1}=\sigma^k$.

 Let $\widetilde G_{\tau,\alpha,k}:K((t;\sigma))\to K((t;\sigma))$ be an anti-automorphism.
  This forces $k=n-1$, $\tau \sigma\tau^{-1}=\sigma^{n-1}$, and $N_{K((t^n))/F((t^n))}(\alpha)=t^{n^2}$ as observed above.

 Conversely, if $\tau\sigma=\sigma^{k}\tau$ for some $k\in \N$ and  $N_{K((t^n))/F((t^n))}(\alpha)=t^{n^2}$ then map $\widetilde  G_{\tau,\alpha,k}$ is an anti-automorphism  on $K((t;\sigma))$ by Theorem \ref{general_antiautomorphism_theorem4}. The fact that it restricts to  $\tau\in\Aut(K)$ yields  $\tau \sigma \tau^{-1}=\sigma^{-1}=\sigma^{n-1}$, so $k=n-1$.

When $n>2$, an anti-automorphism $\widetilde G_{\tau,\alpha,n-1}$ is  of infinite order since it maps $t$ to $(\alpha t)^{-(n-1)}$.

In particular, if $\sigma$ and $\tau$ commute, we must have $\sigma=\sigma^{n-1}$, or  $n-1\equiv 1 \mod n$, which only holds when $n=2$.
And  so if $\widetilde  G_{\tau,\alpha,1}$ is an anti-automorphism then $\sigma$ and $\tau$ commute.
  \end{proof}

  The order of the anti-automorphisms constructed in the above proof can be infinite when $n>2$. For $n=2$ (i.e., when $\sigma$ and $\tau$ commute which is the only time this can happen) and  $\alpha\in K((t^n))^\times$, we have
 $\widetilde  G_{\tau,\alpha,1}^\ell = \widetilde  G_{\tau^\ell,\beta,1}$ with $\beta=\tau^\ell(\alpha)\cdots \tau(\alpha)\alpha$. In particular, the order of  $\widetilde G_{\tau,1,1} $ must be a multiple of the order of $\tau$. However, even when $n=2$, the order of $\widetilde  G_{\tau^\ell,\beta,1}$ depends on the order of $\tau$ and the choice of $\beta$ and is either a multiple of the order of $\tau$ or  infinite.

We now first restrict our considerations to the special case that $\alpha\in K^\times$. Then
$$N_{K((t^n))/F((t^n))}(\alpha)=
N_{K/F}(\alpha)\in F^\times$$
 which is not equal to $t^{n^2}$, so there are no anti-automorphisms $\widetilde G_{\tau,\alpha,n-1}$ on the associative cyclic algebra $(K((t^n))/F((t^n)),\sigma, t^n)$ for any  $\alpha\in K^\times$.
Anti-homomorphisms $\widetilde G_{\tau,\alpha,k}$ with $\alpha\in K^\times$, however, do exist.

\begin{lemma}\label{le:Ktsigmatoitself}
    Let $\tau\in \Aut(K)$ and $k\in \N$. Suppose we only allow  that $\alpha\in K^\times$.  Then the map $\widetilde G_{\tau,\alpha,k}$
     defines a ring anti-homomorphism on $K((t;\sigma))$  if and only if $\tau\sigma\tau^{-1}=\sigma^k$.
     In particular, if $\tau$ and $\sigma$ commute then this occurs if and only if $k\equiv 1 \mod n$.
\end{lemma}

\begin{proof}
The map $\widetilde G_{\tau,\alpha,k}$
  will be an anti-homomorphism if and only if  the canonically induced  map $G_{\tau,\alpha,k}$ between $ K((t;\sigma))$ and its opposite algebra $ K((t;\sigma^{-1}))$ is multiplicative which is the case, if and only if  $G(t)G(c)=G(\sigma(c))G(t)$ for all $c\in K$.
 We compute
$$
G(t)G(c)=\alpha t^{k} \tau(c) = \alpha \sigma^{k}(\tau(c)) t^{k}
$$
and
$$
G(\sigma(c))G(t)=\tau(\sigma(c))\alpha t^{k}.
$$
Since we assume that $\alpha\in K^\times$, this holds for all $c\in K$ if and only if $\tau\sigma=\sigma^{k}\tau$.
 In particular, if $\sigma$ and $\tau$ commute, we must have $\sigma=\sigma^{k}$, or  $k\equiv 1 \mod n$.
\end{proof}

In the proof of \cite[Theorem 4.4]{MST},  a map $\widetilde H_{\tau,1,1}:K((t;\sigma))\to K((t;\sigma))$ is defined that maps $t$ to $t$ via
$$\widetilde H_{\tau,1,1}((\sum_{i} a_i t^i) = \sum_{i}  t^{i}\tau(a_i)=\sum_{i} \sigma^i(\tau(a_i))  t^{i}.
$$
This map is a monomial anti-automorphism of degree one and well-defined exactly when $\tau\sigma\tau^{-1}=\sigma^{-1}$. 

 \begin{lemma}
 Let $\alpha=\sum_{i\geq s}\alpha_ix^i\in K((t^n))^\times$.
 \\ (i) If $N_{K((t^n))/F((t^n))}(\alpha)=t^{n^2}$ then $N_{K/F}(\alpha_s)=1$ and $s=1$.
 \\ (ii) If $\alpha_1\in K^\times$ such that $N_{K/F}(\alpha_1)=1$, then $N_{K((t^n))/F((t^n))}(\alpha_1 t^n)=t^{n^2}$.
 \end{lemma}

 \begin{proof} $(i)$ We have
 $$N_{K((t^n))/F((t^n))}(\alpha)=N_{K/F}(\alpha_s)t^{sn^2}+\dots $$
 where $N_{K/F}(\alpha_s)t^{sn^2}$ is the term with the lowest power of $t^{n^2}$ and the remaining terms all involve higher powers of $t^{n^2}$. The condition
 $N_{K((t^n))/F((t^n))}(\alpha)=t^{n^2}$ hence implies that $N_{K/F}(\alpha_s)=1$ and $s=1$.
 \\ $(ii)$ is clear.
 \end{proof}

We can now  prove a result complementing \cite[Theorem 4.4]{MST}.

\begin{theorem}\label{thm:Patmain}
Let $m>0$, $n>2$ be integers, and let $G_0$ be the metacyclic group generated by $\sigma$ and $\tau\in \Aut(K)$ with $\sigma^n=id$, $\tau^{4m}=id$ and $\tau\sigma\tau^{-1}=\sigma$.
Let $K/F_0$ be a  Galois extension of degree $mn$ with Galois group $G_0$, and let $F={\rm Fix}(\sigma)$.
\\ (i) $K((t;\sigma))$ has a $\tau$-semilinear anti-automorphism given by $G_{\tau,t^n,n-1}$ which is monomial of degree one; indeed for any  $\alpha\in K((t))\setminus K$ such that $\alpha=\alpha_1 t^n$ and $N_{K/F}(\alpha_1)=1$, $\tau$ extends to an anti-automorphism $\widetilde G_{\tau,\alpha,n-1}$.
\\ (ii)
The ring $K((t;\sigma))$ does not admit an involution $f$ such that $\bar f(a)=a$ for all $a\in F_0$. In particular, if $F_0$ has no non-trival automorphism (e.g., $F_0=\mathbb{Q}$ or $\mathbb{Q}_p$), then $K((t;\sigma))$ has no involution at all.
\end{theorem}

\begin{proof}
$(i)$ This follows immediately from the fact  that for any  such $\alpha$, $\tau$ extends to an anti-automorphism $\widetilde G_{\tau,\alpha,n-1}$ on  $K((t;\sigma))$ since  $\tau \sigma\tau^{-1}=\sigma^{n-1}$ and $N_{K((t^n))/F((t^n))}(\alpha_1 t^n)=t^{n^2}$ (cf. Theorem \ref{general_antiautomorphism_theorem4} or Theorem \ref{thm:Patrelvant}).
\\ $(ii)$
This is shown in the proof of \cite[Theorem 4.4]{MST}.
\end{proof}

For general anti-homomorphisms on $K((t;\sigma))$ that restrict to some $\tau\in \Aut(K)$ on $K$ we cannot say much.

\begin{lemma}
 Let $K/F$ be a cyclic Galois field extension of degree $n$ with Galois group generated by $\sigma$.
  Let $\widetilde G:K((t;\sigma))\to K((t;\sigma))$ be an anti-automorphism such that $G|_K=\tau\in \Aut(K)$. Then
  $\tau \sigma\tau^{-1}=\sigma^{n-1}$ and
   $$\widetilde G(t) = \sum_{(n-1)i} t^{-(n-1)i}  \alpha_{(n-1)i}$$
    for some $\alpha_{(n-1)i} \in K$.
\end{lemma}

 \begin{proof}
  Let $\widetilde G:K((t;\sigma))\to K((t;\sigma))$ be an anti-automorphism such that $\widetilde G|_K=\tau\in \Aut(K)$. For every $\tau\in \Aut(K)$ there exists some
  $k\in \{1,\dots, n-1 \}$ with $\gcd (k, n)=1$ such that $\tau \sigma\tau^{-1}=\sigma^k$. In this case, we just saw that since $\widetilde G$ is an anti-automorphism, here $k=n-1$, applying \cite{MST}.
 Suppose $\widetilde G(t) = \sum_{i}\alpha_i t^i$ for some $\alpha_i \in K$. Consider the isomorphism $G$ that is canonically induced by $\widetilde G$ between $K((t;\sigma))=(K((t^n))/F((t^n)),\sigma, t^n)$ and the  algebra $(K((t^n))/F((t^n)),\sigma, (t^n)^{-1})$.
Then we have
\begin{equation} \label{eqn:automorphism_necessity_theorem 1l}
G(tz) = G(t)G(z) = (\sum_{i} \alpha_i t^i ) \tau(z) =
\sum_{i} \alpha_i \sigma^{i}(\tau(z)) t^i \end{equation} and
\begin{equation} \label{eqn:automorphism_necessity_theorem 2l}
G(tz) =
G(\sigma(z)t) = \sum_{i} \tau(\sigma(z)) \alpha_i t^i
\end{equation}
for all $z \in K$. Comparing the coefficients of $t^i$
in \eqref{eqn:automorphism_necessity_theorem 1} and
\eqref{eqn:automorphism_necessity_theorem 2} we obtain
\begin{equation} \label{eqn:automorphism_necessity_theorem 3}
\alpha_i \sigma^{i}(\tau(z)) = \tau(\sigma(z)) \alpha_i
\end{equation}
for all $i $ and all $z \in K$. This implies
$$\alpha_i ( \tau\sigma^i(z) - \tau\sigma(z)  )=0$$
for all $i $ and all $z \in K$. Since
$k=n-1 $,
 $$\tau \sigma\tau^{-1}=\sigma^{n-1}$$
 and $\sigma$ has order $n$, this implies that
 $$\tau \sigma\tau^{-1}\not=\sigma^i$$
 for all  $i \not \equiv \,0, \ldots, n-2 \,{\rm mod} \, n$ so that  $\alpha_i = 0$ for
all  $i \in\{ 0, \ldots, n-2\}$ with $i\not=k$ (we read the $i$'s modulo $n$ here). Therefore
 we get that
 $$\widetilde G(t) = \sum_{(n-1)i}\alpha_{(n-1)i} t^{(n-1)i}$$
  for some $\alpha_{(n-1)i} \in K.$
   \end{proof}

\section{Monomial anti-automorphisms  on generalized cyclic algebras from monomial isomorphism of degree one} \label{sec:nonass}

\subsection{Nonassociative generalized cyclic algebras}

Generalized cyclic  algebras are special examples of central simple algebras which do not have to be crossed products (for interesting examples, see \cite{TH, TH2004, TH2005}).

Let $D$ be a finite-dimensional central division algebra over $C={\rm Cent}(D)$ of degree $r$ and
$\sigma\in {\rm Aut}(D)$ such that $\sigma|_{C}$ has finite order $m$  and fixed field $F={\rm Fix}(\sigma)$.
In this setup,  $C/F$ is automatically a cyclic Galois field extension of
 degree $m$ and $\mathrm{Gal}(C/F) = \langle \sigma |_{C} \rangle$. Then there exists $u \in D^\times$ such that $\sigma^m = i_u$ and $\sigma(u) = u$. These two relations determine $u$ up to multiplication with elements in $F^\times$.

 Let $f(t)=t^m-d\in D[t;\sigma]$ be a two-sided skew polynomial, that means $d\in F^\times$. The associative quotient algebra $D[t;\sigma]/D[t;\sigma](t^m-d)$
 is called a \emph{generalized cyclic algebra} and denoted by $(D,\sigma, d)$ following Jacobson \cite[p.~19]{J96}.

 The generalized cyclic algebra
$(D,\sigma, d)$ is a central simple algebra over $F$ of degree $mr$. Generalized cyclic algebras are special examples of crossed product algebras where  the crossed product is taken using  $D$  and the cyclic group $\langle \sigma\rangle$. Every solvable crossed product division algebra is  a generalized cyclic algebra by Albert.

When $D=C$ and $ f(t)=t^m-d\in F[t;\sigma]$, we obtain the cyclic algebra $(C/F,\sigma,d)$ of degree $m$ as special case.

The classical definition of associative generalized cyclic algebras naturally generalizes to nonassociative algebras. Let $f(t)=t^m-d\in D[t;\sigma]$, $d\in D^\times$.
The nonassociative algebra $D[t;\sigma]/D[t;\sigma](t^m-d)$ over $F$ is called a \emph{(nonassociative) generalized  cyclic algebra of degree $mn$}.
 We denote this algebra also by $(D,\sigma, d)$ and include both the associative ($d\in F$) and the nonassociative ($d\in D\setminus F $) generalized cyclic algebras in this notation.

The generalized cyclic algebra $A=(D, \sigma, d)$, $d\in D^\times$, has dimension $m^2r^2$ over $F$.
If $D=C$ then $(C/F,\sigma,d)$ is a nonassociative cyclic algebra; it is associative if and only if $d\in F^\times$.

The  generalized cyclic algebra $(D, \sigma, d)$ is a division algebra over
$F$ if and only if $f(t)=t^m-d\in D[t;\sigma]$ is irreducible \cite[(7)]{P66}.

 If  $(D, \sigma, d)$ is not associative then ${\rm Nuc}_l(A)={\rm Nuc}_m(A)=D$.

The centralizer of $D$ in $(D,\sigma, d)$ is $C$ when $(D,\sigma, d)$ is associative  \cite[p.~20]{J96} and $D$ when $(D,\sigma, d)$ is not associative.
Since any anti-isomorphism $f$ between two generalized cyclic algebras maps centralizer to centralizer, this means $f|_C\in \Aut(C)$ when $A$ is associative and $f|_D\in \Aut(D)$ when $A$ is not associative.

We thus know (abusing notation we always write $\tau$ for restrictions of $f$ to $D$, $C$ or  $F$): If $f: (D,\sigma, d)\to (D,\sigma, d)$ is an anti-automorphism of  $(D,\sigma, d)$ then  $f|_F=\tau\in \Aut(F)$ since the center $F$ of $(D,\sigma, d)$  is mapped to center.
\\ If  $(D,\sigma, d)$ is not associative then additionally $f|_D=\tau\in \Aut(D)$ and  $\tau|_C\in \Aut(C)$.
\\ If  $(D,\sigma, d)$ is associative then  $f|_C=\tau\in \Aut(C)$.

Let $D$ and $D'$ be two central division algebras over $C$ of degree $r$. Let $\sigma\in \Aut (D)$, $\sigma\in \Aut (D')$,
such that $\sigma|_{C}$ have both finite order $n$,  and fixed field $F={\rm Fix}(\sigma)={\rm Fix}(\sigma')$. Let ${\rm Isom}(D, D')$ denote all the ring isomorphisms between $D$ and $D'$.

\begin{lemma}\label{le:degreek2}
    (i) The elements of ${\rm Isom}(D, D')$ such that $\tau \sigma = \sigma' \tau$ are  isomorphisms $\tau:D\to D'$ such that $\tau|_C\in \Aut(C)$  preserves the subfield $F$ (that is, that restrict to some $\tau\in \Aut(F)$).
    \\ (ii) For every automorphism $\tau\in {\rm Isom}(D, D')$ that restricts  to some $\tau\in \Aut(F)$, there exists a positive integer $k$ with $\gcd(k,m)=1$, such that $\tau|_C \circ \sigma|_C \circ (\tau|_C)^{-1}=(\sigma|_C)^k$.
\end{lemma}

\begin{proof}
$(i)$  As ${\sigma'}|_C$ generates $\Gal(C/F)$, it follows that for any $a\in C$, we have $a\in F$ if and only if $\sigma'(a)=a$.  Now suppose $\tau\in {\rm Isom}(D, D')$ and  $\tau \sigma = \sigma' \tau$.  Then for any $a\in C$, we have
$$
\tau(a) = \tau(\sigma(a))=\sigma'(\tau(a)),
$$
and therefore $a\in F$,  if and only if $\tau(a)\in F$,
so $F$ is preserved by $\tau|_C$.
\\ $(ii)$  Suppose $\tau\in {\rm Isom}(D, D')$  restricts to an automorphism in $\Aut(F)$.  Since $\Gal (C/F)$ is a normal subgroup of $\Aut_F(C)$,  $\tau|_C \circ \sigma|_C \circ (\tau|_C)^{-1}\in \Gal (C/F),$ so there exists a positive integer $k$, such that
$\tau|_C \circ \sigma|_C \circ (\tau|_C)^{-1}=(\sigma|_C)^k$. We have  $\gcd(k,m)=1$ because the order is preserved under conjugation.
\end{proof}

 Define
$$
{\rm Isom}(D, D^{op})_\sigma=\{\tau\in {\rm Isom}(D,D^{op}) \mid \tau \sigma = \sigma \tau \},\quad {\rm Aut}(D)_\sigma=\{\tau\in {\rm Aut}(D) \mid \tau \sigma = \sigma \tau \}.
$$

In the following, for any $\tau\in {\rm Isom}(D, D')$ preserving $F$, we will write  $\tau$-semilinear when we mean $\tau|_F$-semilinear.

\subsection{Monomial  isomorphisms of degree one between nonassociative generalized cyclic algebras}

Let $D$ be a central simple algebra over a field $C$. In this section, we will choose
$\sigma\in {\rm Aut}(D)$ such that $\sigma|_{C}$ has either infinite order, or finite order $n$,  and fixed field $F={\rm Fix}(\sigma)=F^\sigma$. We will initially investigate the isomorphisms between  nonassociative Petit rings of the type $D[t;\sigma]/D[t;\sigma](t^m-d)$.
If $\sigma|_C$ has finite order $n$, then $C/F$ is a cyclic Galois field extension of
 degree $n$ and $\mathrm{Gal}(C/F) = \langle \sigma |_{C} \rangle$. When $n=m$, the algebra $D[t;\sigma]/D[t;\sigma](t^n -d)=(D,\sigma,d)$ is a generalized cyclic algebra.

 \begin{theorem} \label{thm:aut2}
Let $D$ and $D'$ be two central division algebras over $C$ of degree $r$. Let $\sigma\in \Aut (D)$, $\sigma\in \Aut (D')$,
such that $\sigma|_{C}$ have both either infinite order, or finite order $n$,  and fixed field $F={\rm Fix}(\sigma)={\rm Fix}(\sigma')$.
 Let $\tau:D\to D'$ be a ring isomorphism and assume that $\tau\circ\sigma=\sigma'\circ \tau$; that is $D$ and $D'$ are isomorphic algebras over  $F_0=F^\tau$ and  $\tau$ is  $F_0$-linear. Let  $\alpha\in C^\times$.
 \\ (i) For $c,d\in D^\times$,
   $$ G_{\tau,\alpha,1}: D[t;\sigma]/D[t;\sigma](t^m-d)\to D'[t;\sigma']/D'[t;\sigma'](t^m-c),$$
   $$G_{\tau,\alpha,1}(\sum_{i=0}^{m-1}d_it^i)=\sum_{i=0}^{m-1}\tau(d_i) (\alpha t)^i= \sum_{i=0}^{m-1}\tau(d_i) N_i^{\sigma'}(\alpha) t^i $$
    is a $\tau$-semilinear (thus $F_0$-linear) isomorphism between (potentially nonassociative) unital Petit rings if and only if
$$\tau(d) =N^{\sigma'}_{m}(\alpha) c.$$
(ii) If $D[t;\sigma]/D[t;\sigma](t^m-d)$ is proper nonassociative and additionally $n\geq m-1$  every ring isomorphism $G:D[t;\sigma]/D[t;\sigma](t^m-d)\to D'[t;\sigma']/D'[t;\sigma'](t^m-c)$ that restricts to
 $\tau$  is of the type $ G_{\tau,\alpha,1}$.
\end{theorem}

\begin{proof}
We note that if $\tau:D\to D'$ is a ring isomorphism such that $\tau\circ\sigma=\sigma'\circ \tau$ then $\tau|_C\in \Aut(C)$  restricts to some $\tau\in \Aut(F)$ (Lemma \ref{le:degreek2}) and $\tau$ is  $F_0$-linear, where $F_0=F^\tau$, that is $D$ and $D'$ are isomorphic algebras over $F_0$ and $F_0$ is a subfield of $F$.
\\ $(i)$ Let $ G= G_{\tau,\alpha,1}$ be such a ring isomorphism. Then it restricts to $\tau:D\to D'$ by definition. By Lemma \ref{le:degreek2} and the assumption that $\tau\circ\sigma=\sigma'\circ \tau$,
  $\tau|_C\in \Aut(C)$  restricts to some $\tau\in \Aut(F)$, that is $D$ and $D'$ are isomorphic algebras over $F_0$ and $\tau$ is  $F_0$-linear, where $F_0=F^\tau$.

 Since $t^m=t t^{m-1}$ and using that $t^m=d$ in $D[t;\sigma]/D[t;\sigma](t^m-d)$, we have $G(t t^{m-1})=G(t)G(t^{m-1})=G(t)G(t)^{m-1}$, hence we obtain
$G(t^m) = G(d)=\tau(d)$
and
$$
G(t)^m =G(t)G(t)^{m-1}= {\sigma'}^{m-1}(\alpha) \cdots \sigma'(\alpha) \alpha\cdots t^{m} =
N^{\sigma'}_{m}(\alpha)c
$$
 which implies
$$\tau(d) =  N^{\sigma'}_{m}(\alpha) c.$$
Since $\tau\circ\sigma=\sigma'\circ \tau$ we also have $G(ta)=G(t)G(a)$ for all $a\in D$.
Conversely,  if $\tau\circ\sigma=\sigma'\circ \tau$ 
holds
for  some $\alpha\in C^\times$, then the map
$G_{\tau,\alpha,1}:D[t;\sigma]\to D'[t;\sigma']$, $G_{\tau,\alpha,1}(\sum_{i=0}^{{m-1}}d_it^i)=\sum_{i=0}^{m-1}\tau(d_i) N_i^\sigma(\alpha) t^i$  is a
ring isomorphism by \cite[Theorem 3]{Rimmer1978}. Since
$$
G_{\tau,\alpha,1}(t^m-d)=N_m^{\sigma'}(\alpha)t^m- \tau(d)=N_m^{\sigma'}(\alpha)(t^m-N_m^{\sigma'}(\alpha^{-1})\tau(d))=N_m^{\sigma'}(\alpha)(t^m-c),
$$
 this isomorphism $G_{\tau,\alpha,1}$ between the associative skew polynomial rings  canonically induces a ring isomorphism between the nonassociative rings $D[t;\sigma]/D[t;\sigma](t^m-d)$ and $D'[t;\sigma']/D'[t;\sigma']G_{\tau,\alpha,1}(t^m-c)$.
\\ $(ii)$ If the rings are not associative, then every ring isomorphism $G$ between them will map left nucleus onto left nucleus.
 Therefore we always have $G(D)=D'$ as $D$ is the left nucleus  of the algebra $A=D[t;\sigma]/D[t;\sigma](t^m-d)$ and
 $D'$ the left nucleus $A^{op}=D'[t;\sigma']/D'[t;\sigma'](t^m-c)$.
 This forces $G|_D:D\to D'$.  When additionally $n\geq m-1$,  every ring isomorphism $G$ then must be
  of the type $ G_{\tau,\alpha,1}$, analogously as described for instance in the proof of \cite[Theorem 4.2]{Pum2025} or in   \cite{BP}.
\end{proof}

\begin{theorem} \label{thm:autcor}
Let $D$ be a central division algebra over $C$ of degree $r$, and let $(D, \sigma, d)$ and $(D', \sigma', d')$ be two associative  generalized cyclic  algebras. Let $\tau:D\to D'$ be a ring isomorphism and assume that $\tau\circ\sigma=\sigma'\circ \tau$; that is $D$ and $D'$ are isomorphic algebras over  $F_0=F^\tau$ and  $\tau$ is  $F_0$-linear. Let  $\alpha\in C^\times$.
  \\ (i)
   The map
   $$ G_{\tau,\alpha,1}: (D, \sigma, d)\to (D', \sigma', d'),$$
   $$G_{\tau,\alpha,1}(\sum_{i=0}^{m-1}d_it^i)=\sum_{i=0}^{m-1}\tau(d_i) (\alpha t)^i= \sum_{i=0}^{m-1}\tau(d_i) N_i^{\sigma'}(\alpha) t^i $$
    is a $\tau$-semilinear (thus $F_0$-linear) isomorphism between (potentially nonassociative) unital Petit rings if and only if
$$\tau(d) =N^{\sigma'}_{m}(\alpha) d'.$$
(ii) If $G_{\tau,\alpha,k}:(D, \sigma, d)\to (D, \sigma', d')$  is a ring isomorphism for some integer $k$, $1\leq k\leq m-1$, then $k=1$.
 \end{theorem}

\begin{proof}
 $(i)$ follows from Theorem \ref{thm:aut2} $(i)$ since  $(D, \sigma, d)=D[t;\sigma]/D[t;\sigma](t^m-d)$.
 \\ $(ii)$
 If such a $G=G_{\tau,\alpha,k}$ exists then $G(td)=G(\sigma(d)t)$ implies $\alpha t^k \tau(d) = \tau(\sigma(d)) \alpha t^k$ for all $d\in D$, whence we must have
$$
{\sigma'}^k(\tau(d)) = \tau(\sigma(d))
$$
for all $d\in D$.  By assumption, $\tau\circ \sigma=\sigma'\circ \tau$, so we get that $G(td)=G(\sigma(d)t)$ is equivalent to ${\sigma'}^k = \sigma$, which in turn is equivalent to $k \equiv 1 \mod m$, since both $\sigma$ and $\sigma'$ have order $m$. This is not possible for all $k$ such that $2\leq k\leq m-1$.
\end{proof}

\begin{corollary}\label{cor:maingeneralized1}
Let  $D$ be a central division algebra over $C$ and let $(D, \sigma, d)$  be an associative  generalized cyclic  algebra.
Let $\tau:D\to D^{op}$ be a ring isomorphism and assume that $\tau\circ\sigma=\sigma\circ \tau$. Put $F_0=F^\tau=\{a\in F\,|\,\tau(a)=a \}$.
Then the map $ G_{\tau,\alpha,1}: (D,\sigma,d)\to (D^{op},\sigma,d^{-1}),$
   $$G_{\tau,\alpha,1}(\sum_{i=0}^{m-1}d_it^i)=\sum_{i=0}^{m-1}\tau(d_i) (\alpha t)^i= \sum_{i=0}^{m-1}\tau(d_i) N_i^\sigma(\alpha) t^i $$
    is a $\tau$-semilinear (thus $F_0$-linear) isomorphism
     if and only if
$$N_{C/F}(\alpha) = \tau(d) d.$$
\end{corollary}

This follows as a special case from Theorem \ref{thm:aut2}.

\subsection{Monomial anti-automorphisms of degree $m-1$}

Our results from the previous section now immediately imply that we  obtain $\tau$-semilinear (i.e. $F_0$-linear)  anti-automorphisms of an associative  generalized cyclic  algebra $(D, \sigma, d)=D[t;\sigma]/D[t;\sigma](t^m-d)$ which are again canonically induced by ring automorphisms $ G_{\tau,\alpha,1}:D[t;\sigma]\to D[t;\sigma]$  of the twisted polynomial ring  which restrict to some $\tau\in {\rm Aut}(D)_\sigma$.

In order to do so, we first establish how the opposite algebra of a Petit algebra $D[t;\sigma]/ D[t;\sigma](t^m-d)$ looks like. The canonical anti-isomorphism
 $\mathbb{S}_{t^m-d}=D[t;\sigma]\to D^{op}[t;\sigma^{-1}], $ $\sum_{i=0}^\ell d_it^i \mapsto \sum_{i=0}^\ell t^i d_i $, induces an anti-isomorphism between the Petit rings $\mathbb{S}_{t^m-d}$ and $\,_{t^m-d}\mathbb{S}$;
$$D[t;\sigma]/[t;\sigma]  (t^m-d) \to D^{op}[t;\sigma^{-1}]/ (t^m-d) D^{op}[t;\sigma^{-1}],\quad  \sum_{i=0}^{m-1} d_it^i \mapsto \sum_{i=0}^{m-1} t^i d_i = \sum_{i=0}^{m-1} \sigma^{-i}(d_i)t^i  ,$$
which shows that
$$\big(D[t;\sigma]/ D[t;\sigma](t^m-d)\big)^{op} \cong D^{op}[t;\sigma^{-1}]/ (t^m-d) D^{op}[t;\sigma^{-1}].$$
In particular, for an associative generalized cyclic algebra $(D,\sigma,d)$, where $d\in F^\times$, we have
$$(D,\sigma, d)^{op}\cong (D,\sigma^{-1},d), \quad
\sum_{i=0}^{m-1} d_it^i \mapsto \sum_{i=0}^{m-1} \sigma^{-i}(d_i)t^i  .$$
Moreover, when $D[t;\sigma^{-1}]/ D[t;\sigma^{-1}](t^m-d)$ is an associative algebra, $t$ has a well-defined inverse $t^{-1}$ given by $t^{m-1}d^{-1}$, and
$$ D[t;\sigma]/ D[t;\sigma](t^m-d^{-1}) \to  D^{op}[t;\sigma^{-1}]/ D^{op}[t;\sigma^{-1}](t^m-d)
,\quad  \sum_{i=0}^{m-1} d_it^i \mapsto \sum_{i=0}^{m-1}d_i t^{-i} ,$$
defines an $F$-algebra isomorphism. We will use the notation $t^{-1}$, using that
reading $t^{-i}$ as short for $t^{m-i}d^{-1}=d^{-1}t^{m-i}$.
In particular, for an associative generalized cyclic algebra $(D,\sigma,d)$, we therefore have
$$(D,\sigma, d^{-1})\cong (D,\sigma^{-1},d).$$
We conclude that for associative generalized cyclic algebras,
$$(D,\sigma, d)^{op}\cong (D,\sigma,d^{-1}),  \quad
\sum_{i=0}^{m-1} d_it^i \mapsto \sum_{i=0}^{m-1} \sigma^{-i}(d_i)t^{-i} ,$$
so we can apply our previous results.

\begin{theorem} \label{thm:maingeneralizedantiautcor} 
Let $D$ be a central division algebra over $C$, $ (D, \sigma, d)$ be an associative generalized algebra, and assume that $\tau\in {\rm Isom}(D,D^{op})_\sigma$ is  $F_0$-linear.
\\ (i)
  There exists a $\tau$-semilinear anti-automorphism $\widetilde G_{\tilde\tau,\alpha,1}$ on the $F$-algebra $ (D, \sigma, d)$ defined via
$$\widetilde G_{\tilde\tau,\alpha,1}(\sum_{i=0}^{m-1}d_it^i)=\sum_{i=0}^{m-1} (\alpha t)^{-i} \tilde \tau(d_i)= \sum_{i=0}^{m-1} \sigma^{-i}(\tau(d_i)) N_i^{\sigma^{-1}}(\alpha) t^{-i},$$
 if and only if  there exists  some $\alpha\in C^\times$ such that
$$N_{C/F}(\alpha)= \tau(d)d.$$
(ii) Let  $k$ be an integer, $1\leq k\leq m-1$. If $\widetilde G_{\tilde \tau,\alpha,k}$ is an $F_0$-linear anti-automorphism  for a suitable $\alpha\in C^\times$ then $k=1$.
\end{theorem}

The anti-automorphism $\widetilde G_{\tau,\alpha,k}$ is monomial of degree $m-k$, since $t^{-k}=d^{-1}t^{m-k}$.

\begin{proof}
$(i)$
There exist  $F_0$-linear isomorphisms
 $$ G_{\tau,\alpha,1}: D[t;\sigma]/D[t;\sigma](t^n-d)\to D^{op}[t;\sigma]/D^{op}[t;\sigma](t^n-d^{-1}),$$
    for suitable $\alpha\in C^\times$ and $\tau \in {\rm Isom}(D,D^{op})_\sigma$, if and only if $\tau(d) = N^\sigma_{n}(\alpha) d^{-1}$
which is the same as
   $N^\sigma_{m}(\alpha)= \tau(d)d$, and since $\alpha\in C^\times$ and $\sigma|_C\in {\rm Gal}(C/F)$ has order $m$ and generates the Galois group, the left hand side becomes  $N_{C/F}(\alpha)$.
   Since
$$(D[t;\sigma]/D[t;\sigma](t^m-d))^{op}=D^{op}[t;\sigma^{-1}]/ (t^m-d) D^{op}[t;\sigma^{-1}] \cong D^{op}[t;\sigma]/D^{op}[t;\sigma](t^m-d^{-1}),$$
 the isomorphism $G_{\tau,\alpha,1}$ induces an  $F_0$-linear anti-automorphism $\widetilde G_{\tau,\alpha,1}$ on the associative algebra $ (D, \sigma, d)=D[t;\sigma]/D[t;\sigma](t^m-d)$, given by
$$
\widetilde G_{\tau,\alpha,1}(\sum_{i=0}^{m-1}d_it^i)=\sum_{i=0}^{n-1}(\alpha t)^{-i}\tau(d_i) = \sum_{i=0}^{m-1} N_i^{\sigma^{-1}}(\alpha) t^{-i} \tau(d_i)
= \sum_{i=0}^{m-1} \sigma^{-i}(\tau(d_i)) N_i^{\sigma^{-1}}(\alpha) t^{-i} .$$
\\ $(ii)$ Let now $\widetilde G$ be any $F_0$-linear anti-automorphism which restricts to some  anti-automorphism  $\tilde \tau$ on $D$. Then $\widetilde  G$ defines an $F_0$-algebra isomorphism $G$ between the $F_0$-algebra $D[t;\sigma]/D[t;\sigma](t^m-d)$ and its opposite algebra, that restricts to $\tau\in {\rm Isom}(D,D^{op})$.

Since $\tau\in {\rm Isom}(D,D^{op})_\sigma$ and  $\sigma$ has order $m$, we know that
every such $\tau$-linear  ring isomorphism has the form $ G_{\tau,\alpha,1}$ by Theorem \ref{thm:autcor} $(ii)$, therefore every $\tau$-semilinear (thus $F_0$-linear) anti-automorphism  has the form $G=\widetilde G_{\tau,\alpha,1}$.
\end{proof}

 \begin{lemma}\label{le:general}
(i)
For all positive integers $\ell$, we have $(\widetilde G_{\tau,\alpha,1})^\ell=\widetilde G_{\tau^\ell,\tau^\ell(\alpha,1)\cdots \tau(\alpha)\alpha,1}=id$  if and only if $\tau^\ell=id$ and
 $\tau^\ell(\alpha)\cdots \tau(\alpha)\alpha=1.$
 \\ (ii) $\widetilde G_{\tau,\alpha,1}$ is an involution if and only if $\tau^2=id$ and
 $ \tau(\alpha)\alpha=1.$
 \\ (iii)  If the maps $\widetilde G_{\tau_i,\alpha_i,1}$  are  anti-automorphisms of $D[t;\sigma]/D[t;\sigma](t^m-d)$ 
  then for odd $s$,
 $$\widetilde G_{\tau_s,\alpha_s,1}\circ\cdots \circ  \widetilde G_{\tau_1,\alpha_1,1}=\widetilde G_{\tau_s\circ\cdots \circ \tau_1,\beta,1}$$
 with
 $$\beta=\tau_s(\dots \tau(\alpha_1)\dots ) \tau_{s-1}(\dots \tau(\alpha_2)\dots )\cdots \tau_1(\alpha_{s-1}) \alpha_s$$
 is an anti-automorphism of $D[t;\sigma]/D[t;\sigma](t^m-d)$ 
 which is $(\tau_s\circ\cdots \circ \tau_1)$-semilinear.
  \end{lemma}

The proof is analogous to the proof of Lemma \ref{le:1}, since $\alpha\in C^\times$. The result can be used to check when $\widetilde G_{\alpha,\tilde\tau,1}$  has infinite order.

\begin{corollary}
Let  $(D,\sigma,d)$ be an associative generalized cyclic algebra of degree $rm$ and let $\tau \in {\rm Isom}(D,D^{op})_\sigma$. Let $F_0=F^\tau$ be the fixed field of $\tau$.
Then there exists an  anti-automorphism of the first kind
 $\widetilde G_{\alpha,\tilde\tau,1}$ on $(D,\sigma,d)$
  if and only if there exists  some $\alpha\in C^\times$ and some $\tau\in {\rm Isom}(D,D^{op})_\sigma$,  such that $\tau|_F=id$ and
  $$N_{C/F}(\alpha)= \tau(a)a.$$
  In particular, $(D,\sigma,d)$ permits an  anti-automorphism of the first kind
 $\widetilde G_{\alpha,\tilde\tau,1}$ if and only if there exists  some $\alpha\in C^\times$ and some $\tau\in {\rm Isom}(D,D^{op})_\sigma$,  such that $\tau|_F=id$ and
  $$N_{C/F}(\alpha)= d^2.$$
\end{corollary}

\begin{corollary}\label{cor:inv}
Let  $(D,\sigma,d)$ be an associative generalized cyclic algebra of degree $rm$ and let $\tau \in {\rm Isom}(D,D^{op})_\sigma$. Let $F_0=F^\tau$ be the fixed field of $\tau$.
Then there exists an $F_0$-linear involution $\widetilde G_{\tau,\alpha,1}$ on $(D,\sigma,a)$
 if and only if $\tau^2=id$ and there exists  some $\alpha\in C^\times$ such that
$$N_{C/F}(\alpha)=a\tau(a) \text{ and } \tau(\alpha)\alpha=1.$$
\end{corollary}

\begin{proof}
By Theorem \ref{thm:maingeneralizedantiautcor}, such an involution exists if and only if there exists  some $\alpha\in C^\times$ such that
$N^\sigma_{m}(\alpha)= a\tau(a).$ By Lemma \ref{le:general}, $\widetilde G_{\tau,\alpha,1}$ is an involution if and only if $\tau^2=id$ and
 $ \tau(\alpha)\alpha=1.$
\end{proof}

\section{Monomial anti-automorphisms  on generalized cyclic algebras from monomial isomorphism of degree $k$}

When do monomial $\tau$-semilinear anti-automorphisms on associative generalized cyclic algebras exist, where $t$ gets mapped to $(\alpha t^k)^{-1}$ for some $k>1$ and $\alpha \in C^\times$?

\subsection{Monomial  isomorphisms of degree $k>1$ between nonassociative generalized cyclic algebras}

We start looking at isomorphisms between more general Petit algebras. Let $D$ and $D'$ be two central division algebras over $C$ of degree $r$.
For a ring isomorphism $\tau:D\to D'$ that preserves $F$ (i.e., $\tau(F)=F$) define $F_0=F^\tau=\{a\in F\,|\,\tau(a)=a \}$

\begin{lemma}\label{L:k=1}
Let $D$ and $D'$ be two central division algebras over $C$ of degree $r$. Let $\sigma\in \Aut (D)$, $\sigma\in \Aut (D')$,
such that $\sigma|_{C}$ have both finite order $n$,  and assume that $F={\rm Fix}(\sigma)={\rm Fix}(\sigma')$.
 Let $\tau:D\to D'$ be a ring isomorphism.
      If for some  $\alpha\in C^\times$ and degree $k>1$, the map $G_{\tau,\alpha,k}:D[t;\sigma]/D[t;\sigma](t^m-d)\to D'[t;\sigma']/D[t;\sigma'](t^m-e)$ is a monomial homomorphism, then $$ \tau\sigma \tau^{-1}={\sigma'}^k;$$
       if it is an isomorphism then $\gcd(k,m)=1$.
\end{lemma}

This generalizes \cite[Lemma 4.3]{NevPum2025}. Note that $ \tau\sigma \tau^{-1}={\sigma'}^k$ implies that $\tau$ preserves $F$.

\begin{proof}
    If $G=G_{\tau,\alpha,k}:D[t;\sigma]/D[t;\sigma](t^m-d)\to D'[t;\sigma']/D[t;\sigma'](t^m-e)$ is a homomorphism, then $G(td)=G(\sigma(d)t)$ implies $\alpha t^k \tau(a) = \tau(\sigma(a)) \alpha t^k$ for all $a\in D$, which means
$$
{\sigma'}^k(\tau(a)) = \tau(\sigma(a))
$$
for all $a\in D$.  Thus $G(ta)=G(\sigma(a)t)$ is equivalent to ${\sigma'}^k\tau = \tau\sigma$.

If $G$ is to be an isomorphism, there must exist an inverse homomorphism $H$.
Evidently, we must have $H(a)=\tau^{-1}(a)$ for all $a\in D$.  Since $G$ is monomial so is its inverse, and thus we may assume $H(t)=\beta t^\ell$ for some $\beta \in C^\times$ and $1\leq \ell < m$.  Since $H(G(t))=t$ we must have $\ell k \equiv 1 \mod m$ and in particular $\gcd(k,m)=1$.
Moreover, if $\alpha$ were not invertible, then $t^k$ could not lie in the image of $G$.
\end{proof}

As direct consequence, we recover Theorem \ref{thm:maingeneralizedantiautcor} $(ii)$ as part $(i)$ of the next corollary.

\begin{corollary}
Let $\alpha \in C^\times$.
\\ (i) When $\tau$ and $\sigma$  commute, then there do not exist any  monomial anti-automorphisms $\widetilde G_{\tilde \tau,\alpha,k}$ of degree $m-k$ with $k>1$  on the algebra $D[t;\sigma]/D[t;\sigma](t^m-d)$, in particular thus on any generalized cyclic algebra $(D,\sigma,d)$.
\\ (ii) When $k\in \{2,\dots, m-1 \}$ and $\sigma^k\tau \not= \tau\sigma$, then there do not exist any  monomial anti-automorphisms $\widetilde G_{\tilde\tau,\alpha,k}$ of degree $m-k$ with $k>1$ on the algebra $D[t;\sigma]/D[t;\sigma](t^m-d)$. In particular there do not exist any  monomial anti-automorphisms $\widetilde G_{\tilde\tau,\alpha,k}$ of degree $k>1$  on any generalized cyclic algebra $(D,\sigma,d)$.
\end{corollary}

\begin{theorem}\label{thm:G}
Let $D$ and $D'$ be two central division algebras over $C$ of degree $r$. Let $\sigma\in \Aut (D)$, $\sigma\in \Aut (D')$,
such that $\sigma|_{C}$ have both finite order $n$,  and fixed field $F={\rm Fix}(\sigma)={\rm Fix}(\sigma')$.
 Let $\tau:D\to D'$ be a ring isomorphism. 
Suppose  $\alpha \in C^\times$.
Then for any $2\leq k < m$, the map
$$
 G_{\tau,\alpha,k}:D[t;\sigma]/D[t;\sigma](t^m-d)\to D'[t;\sigma']/D'[t;\sigma'](t^m-e)
 $$
 is a ring homomorphism  if and only if all of the following conditions hold:
 \begin{enumerate} \setlength{\itemsep}{0pt}
        \item \label{C:k=1} $ \tau\sigma \tau^{-1}={\sigma'}^k$;
        \item \label{C:n|m} $n\mid m$; 
        \item \label{C:inF} $d\in {\rm Fix}({\sigma'}^k)$;
        \item \label{C:norm} $N_{n}^{{\sigma'}}(\alpha)^{m/n}(e)^{mk/n}=\tau(d)$. 
    \end{enumerate}
 It is an isomorphism if and only if, in addition, we have
    \begin{enumerate} \setcounter{enumi}{4}
        \item \label{C:relprime} $\gcd(k,m)=1$.
  \end{enumerate}
\end{theorem}

 This generalizes parts of \cite[Theorem 4.4]{NevPum2025}. Note that $ \tau\sigma \tau^{-1}={\sigma'}^k$ implies that $\tau$ preserves $F$ and hence $F_0=F^\tau$ is well-defined and  a subfield of $F$.

\begin{proof}
Let $\tau:D\to D'$ be a ring isomorphism. Fix $\alpha\in C^\times$ and  $2\leq k<m$.  Suppose first that the map $G=G_{\tau,\alpha,k}:D[t;\sigma]/D[t;\sigma](t^m-d)\to D'[t;\sigma']/D'[t;\sigma'](t^m-e)$ defined by the rule
\begin{equation}\label{E:formulaG(t)}
G( \sum_{i=0}^{m-1}a_it^i ) = \sum_{i=0}^{m-1}\tau(a_i) (\alpha t^k)^i
\end{equation}
is a well-defined ring homomorphism.
Then Lemma~\ref{L:k=1} immediately yields Conditions (\ref{C:k=1}) and (\ref{C:relprime}).

For $G$ to be well-defined, we require for all $0<s <m$ that $G(t^s)=G(t)^s$ is well-defined.  Since $k>1$, the minimal value of $r$ such that $rk\geq m$ satisfies $0<r<m$ and thus by \cite[Theorem 3.3]{NevPum2025}, we conclude that $\alpha e = {\sigma'}^{m}(\alpha){\sigma'}^k(e)$.  Conversely, \cite[Theorem 3.3]{NevPum2025} further assures us that this hypothesis implies all powers of $\alpha t^k$ are well-defined so \eqref{E:formulaG(t)} gives a well-defined map. (We note that the fact that $D$ is not commutative does not affect the proof of \cite[Theorem 3.3]{NevPum2025}, so that the statements hold in our setting as well).

Moreover, by \cite[Theorem 3.3]{NevPum2025}  the element $G(t^s)=G(t)^s \in D'[t;\sigma]/D'[t;\sigma](t^m-e)$ will be a monomial of degree $[sk]_{m}$ for any $0<s<m$, where  $[sk]_{m}$ denotes the residue of $sk$ mod $m$. In order for $G$ to further preserve the relation $t^sa={\sigma}^s(a)t^s$ for all $a\in D$ we require $G(t^s)\tau(d) = \tau({\sigma}^s(d))G(t^s)$, which will hold if and only if
$$
{\sigma'}^{[sk]_{m}}(\tau(a))=\tau(\sigma^{s}(a))
$$
for all $a\in D$.

Since $ \tau\sigma \tau^{-1}={\sigma'}^k$ by Lemma~\ref{L:k=1} we also have  by induction
$ \tau\sigma^s \tau^{-1}={\sigma'}^{ks}$, since $\sigma'$ has order $n$ this is the same as
$$
{\sigma'}^{[sk]_{n}}(\tau(a))=\tau(\sigma^{s}(a))
$$
for all $0<s<m$ and for all $a\in D$. Thus ${\sigma'}^{[sk]_{m}}(\tau(a))=\tau(\sigma^{s}(a))$  holds if and only if $n\mid m$.

 Thus a second necessary condition is (\ref{C:n|m}),  $n\mid m$, and it in return implies $G(t^s)\tau(a) = \tau(\sigma^s(a))G(t^s)$ for all $0<s<m$ and $a\in D$.

Since $n\mid m$ the relation $\alpha e = {\sigma'}^{m}(\alpha){\sigma'}^k(e)$
 simplifies to $e ={\sigma'}^k(e)$, 
 yielding condition (\ref{C:inF}).

Next, $G$ must satisfy $G(t^{m}-d)=0$.  By the above, we may evaluate $G(t^{m})=G(t)^{m}$ using the formula
$$
(\alpha t^k)^s = N_{s}^{{\sigma'}^k}(\alpha)\prod_{i=1}^{\lfloor \frac{sk}{{n}}\rfloor} {\sigma'}^{sk-i{n}}(e)t^{[sk]_{n}}
$$
where $\lfloor \frac{sk}{n}\rfloor=j$ when $j{n}\leq sk<(j+1){n}$ and $[sk]_{n} = sk-\lfloor \frac{sk}{n}\rfloor n$ denotes the residue of $sk$ mod $n$ in the interval $[0,n-1]$. With  ${\sigma'}^{n}=id$ and $n\mid m$,  this yields 
$$
G(t)^{m} = (\alpha t^k)^m = N_{m}^{{\sigma'}^k}(\alpha)\prod_{i=1}^{\lfloor \frac{mk}{{n}}\rfloor} {\sigma'}^{mk-i{n}}(e)t^{[mk]_{n}} =
N_{n}^{{\sigma'}}(\alpha)^{m/n}(e)^{mk/n}.$$
from \cite[Theorem 3.3]{NevPum2025}.

Since $C/F$ is  a cyclic Galois field extension of
 degree $n$ and $\mathrm{Gal}(C/F) = \langle {\sigma'} |_{C} \rangle$, ${\sigma'}|_C\in \Aut(C)$ has order $n$.

Because of $G(t^{m})=G(d)$, this is equivalent to the condition (\ref{C:norm})
$$
N_{n}^{{\sigma'}}(\alpha)^{m/n}(e)^{mk/n}=\tau(d).
$$
Note that if $G$ is an isomorphism, its inverse map must be a homomorphism.  If $G^{-1}(t)=\beta t^\ell$, then $\ell \equiv 1 \mod n$; in fact, $\ell$ is simply the inverse of $k$ mod $m$ (as guaranteed by condition (\ref{C:relprime})).

Finally, if $k,m,n,d,\tau$ satisfy the first four conditions, then the $F^\tau$-linear map $G$ given by \eqref{E:formulaG(t)} is a well-defined homomorphism of $F^\tau$-modules, and respects the relations $G(t^{m})=G(d)$ and $G(t^su)=G(\sigma^s(u)t^s)$ for all $0<s<m$ and all $u\in D$.  Therefore in particular it satisfies $G(p)G(q)=G(pq)$ and $G(p+q)=G(p)+G(q)$ for all $p,q\in D[t;\sigma]/D[t;\sigma](t^m-d)$, which implies it is an homomorphism of $F^\tau$-algebras.  The final condition ensures it is an isomorphism.
\end{proof}

As a consequence, we obtain another main result.

\begin{theorem}\label{thm:main4}
Suppose $\sigma$ has order $m$. Let $\tau:D\to D^{op}$ be a ring isomorphism.
\\ (i) Let $\alpha \in C^\times$.  Then for any $2\leq k < m$, the map
$$
 G_{\tau,\alpha,k}:(D,\sigma,d)\to (D^{op},\sigma,d^{-1})
 $$
 is a ring isomorphism  if and only if all of the following conditions hold:
 \begin{enumerate} \setlength{\itemsep}{0pt}
        \item \label{C:k=1c} $ \tau\sigma \tau^{-1}=\sigma^k$;
        \item \label{C:inFc} $d\in F^\times$; 
        \item \label{C:normc} $N_{C/F}(\alpha)=\tau(d)d^{k}$;
        \item \label{C:relprimec} $\gcd(k,m)=1$.
  \end{enumerate}
   (ii) If $G_{\tau,\alpha,k}:(D,\sigma,d)\to (D,\sigma,d)^{op}$ is a ring isomorphism for any $2\leq k < m$, then $(D,\sigma,d)$ is associative.
\end{theorem}

In particular, if $\tau$ can be extended to $G_{\tau,\alpha,k}$ then $\tau$ preserves $F$ (the center of $(D,\sigma,d)$ and $ (D^{op},\sigma,d^{-1}$).

\begin{proof}
$(i)$
We use Theorem \ref{thm:G}.
Since $\sigma$ has finite order $m$,  $C/F$ is a cyclic Galois field extension of
 degree $m$ and $\mathrm{Gal}(C/F) = \langle \sigma |_{C} \rangle$. Since $\gcd(k,m)=1$, we know that $\mathrm{Gal}(C/F) = \langle \sigma^k |_{C} \rangle$, so $\sigma^k |_{C} $ also has order $m$ and fixed field $F$. Hence we know  (\ref{C:inF}) in Theorem \ref{thm:G} becomes (\ref{C:inFc}). The rest is straightforward.
 \\ $(ii)$  If such a $G=G_{\tau,\alpha,k}$ as in $(i)$ exists then $d\in F^\times$ which yields the assertion.
\end{proof}

 \begin{lemma}\label{L:Ktsigmatoitself}
    Let $\tau:D\to D'$ be a ring isomorphism, $k$ a positive integer,  and $\alpha\in C^\times$. Let $\sigma\in \Aut(D)$ and $\sigma'\in \Aut(D')$. Then the map generated by $G(c)=\tau(c)$ for all $c\in D$ and $G(t)=\alpha t^k$ defines a ring homomorphism  from $D[t;\sigma]$ to $D'[t;\sigma']$ if any only if $\tau\sigma\tau^{-1}={\sigma'}^k$.
\end{lemma}

A special case of this result can be found in \cite[Theorem 3]{Rimmer1978}.

\begin{proof}
The map $G:D[t;\sigma]\to D[t;\sigma']$ generated by $G(c)=\tau(c)$ for all $c\in D$ and $G(t)=\alpha t^k$ is given by
$$
G(\sum_{i=0}^{\ell} a_i t^i) = \sum_{i=0}^\ell \tau(a_i) N_i^\sigma(\alpha) t^{ik},
$$
for any $\ell\in \mathbb{N}$ and $a_i\in D$.
This map
is multiplicative if and only if  $G(t)G(c)=G(\sigma(c))G(t)$ for all $c\in D$.
 We compute
$$
G(t)G(c)=\alpha t^k \tau(c) = \alpha {\sigma'}^k(\tau(c)) t^k
$$
and
$$
G(\sigma(c))G(t)=\tau(\sigma(c))\alpha t^k.
$$
This holds for all $c\in D$ if and only if $\tau\sigma={\sigma'}^k\tau$.
\end{proof}

 These ring isomorphisms thus canonically induce the ring isomorphisms in the cases treated above.

Thus monomial anti-automorphisms of degree greater than one also exist on associative generalized cyclic algebras.

\subsection{Monomial anti-automorphisms of degree $m-k$ with $k>1$ on generalized cyclic algebras}

Monomial anti-automorphisms on associative generalized cyclic algebras mapping $t$ to $(\alpha t^k)^{-1}$ with $k>1$ and restricting to $\tau\in \Aut(D)$ do not exist when
$\tau\in {\rm Isom}(D,D^{op})_\sigma$ (Theorem \ref{thm:maingeneralizedantiautcor}).

\begin{theorem}\label{cor:tilde G}
Let $(D,\sigma,d)$ be an associative  generalized cyclic algebra (i.e. $\sigma$ has order $m$). Let $\tau:D\to D^{op}$ be a a ring isomorphism and $F_0=F^\tau$.
  Let $\alpha \in C^\times$.  Then for any $2\leq k < m$, the map
$\widetilde  G_{\tilde\tau,\alpha,k}$
 is a $\tau$-semilinear (i.e. $F_0$-linear) anti-automorphism on  $(D,\sigma,d)$, if and only if all of the following conditions hold:
 \begin{enumerate} \setlength{\itemsep}{0pt}
        \item \label{C:k=1c} $ \tau\sigma \tau^{-1}=\sigma^k$;
        \item \label{C:normc} $N_{C/F}(\alpha)=\tau(d)d^{k}$.
        \item \label{C:relprimec} $\gcd(k,m)=1$.
  \end{enumerate}
\end{theorem}

This result follows directly from  Theorem \ref{thm:main4}.

\begin{theorem}\label{thm:maingeneralizedantiaut2}
Let $\tau:D\to D^{op}$ be a a ring isomorphism and $F_0=F^\tau$.
Let  $\alpha\in C^\times$.
 Let  $(D,\sigma,d)$ be a generalized associative cyclic algebra of degree $rm$ ($d\in F^\times$).
\\ (i)
There exists a $\tau$-semilinear anti-automorphism $\widetilde G_{\tilde\tau,\alpha,1}$ on $(D,\sigma,d)$ defined via
$$\widetilde G_{\tilde\tau,\alpha,1}(\sum_{i=0}^{m-1}d_it^i)=\sum_{i=0}^{m-1} (\alpha t)^{-i} \tilde \tau(d_i),$$
 if and only if $\tau\sigma=\sigma\tau$ and
 there exists  some $\alpha\in C^\times$ such that
$$N_{C/F}(\alpha)= \tau(d)d.$$ 
 (ii)  For any $2\leq k < m$, the map
$\widetilde  G_{\tilde\tau,\alpha,k}$
 is a $\tau$-semilinear anti-automorphism on  $(D,\sigma,d)$, if and only if
 $$ \tau\sigma \tau^{-1}=\sigma^k,\quad N_{C/F}(\alpha)=\tau(d)d^{k}  \text{ and } \gcd(k,m)=1.$$
 \end{theorem}

\begin{remark}\label{re:Timo}
The isomorphism criterium in the fist special case mentioned in \cite[Section 3]{TH2007} is indeed a norm condition, but not the one presented there. We correct it as follows. For two cyclic algebras $(K_1/F,\sigma_1,a_1)$ and $(K_2/F,\sigma_2,a_2)$ where we have an $F$-isomorphism $\tau:K_1\to K_2$ such that $\tau \sigma_1\tau^{-1}=\sigma_2^k$ for some $k\in \{2,\dots, n-1 \}$ with $\gcd (k, n)=1$, there exists an isomorphism $G_{\tau,\alpha,k}$ that extends $\tau$ if and only if
$$
N_{K_2/F}(\alpha)a_2^k=a_1.
$$
We can then explicitly construct such isomorphisms as suitable $G_{\tau,\alpha,k}$. Along the same lines, in the ensuing general case on the same page, for two generalized cyclic algebras $(D_1,\sigma_1,d_1)$ and $(D_2,\sigma_2,d_2)$ over $F$, where we have a $C$-isomorphism $\tau:D_1\to D_2$ such that $\tau \sigma_1\tau^{-1}=\sigma_2^k$ for some $k\in \{2,\dots, n-1 \}$ with $\gcd (k, n)=1$, we also obtain a norm condition for the existence of an extension   $G_{\tau,\alpha,k}$  of $\tau $ and thus in particular an explicit construction of  isomorphisms $G_{\tau,\alpha,k}:(D_1,\sigma_1,d_1)\to (D_2,\sigma_2,d_2)$ that extend $\tau$. The norm condition is $ N_{C/F}(\alpha)d_2^{k}=\tau(d_1)$.
\end{remark}

 \section{Anti-automorphisms on the division ring of twisted Laurent series $D((t;\sigma))$}

 Let $D$ be a finite-dimensional central division algebra over $C$ of degree $r$ and
$\sigma\in {\rm Aut}(D)$ such that $\sigma|_{C}$ has finite order $m$  and fixed field $F={\rm Fix}(\sigma)$.
Then  $C/F$ is  a cyclic Galois field extension of
 degree $m$ and $\mathrm{Gal}(C/F) = \langle \sigma |_{C} \rangle$. There exists $u \in D^\times$ such that $\sigma^m = i_u$ is an inner automorphism and $\sigma(u) = u$. These two relations determine $u$ up to multiplication with elements from $F^\times$.

 Put $x=u^{-1}t^m$. The set $D((t;\sigma))$ of all formal power series with coefficients in $D$ is the generalized cyclic division algebra $(D((x)),\sigma, ux)=(D((x)),\sigma, t^m)$ of degree $rm$ over its center $F((x)) $. If $D$ has a maximal subfield that is Galois over $F$ then $D((t;\sigma))$ is a crossed product algebra. For instance, when $D$ is a $p$-algebra, then $D((t;\sigma))$ is a cyclic crossed product algebra. However, there exist many examples, where $D((t;\sigma))$ is not a crossed product algebra \cite{TH}.

  Every isomorphism $\tau\in {\rm Isom}(D, D^{op})$  restricts to an automorphism on the center $C$. If $\tau|_C $ preserves $F$, then there exists a positive integer $k$, $\gcd(k,m)=1$, 
  such that $\tau|_C \circ \sigma \circ (\tau|_C )^{-1}=\sigma^k$ (Lemma \ref{le:degreek2}). Note that not all  $\tau\in {\rm Isom}(D, D^{op})$ will preserve $F$.

  By a canonical generalization of \cite[Proposition 4.3]{MST}, every anti-automorphism $f$ on $D((t;\sigma))$ that restricts to an automorphism of $C$ that leaves $F$ invariant  satisfies  
  $$\bar f|_C \circ \sigma \circ (\bar f|_C)^{-1}=\sigma^{-1},$$
   where $\bar f$ denotes the canonical anti-automorphism on the residue algebra $D$ that is induced by $f$.
This directly implies that for every anti-automorphism $f$ on $D((t;\sigma))$ that restricts to some anti-automorphism $\tau$ on $D$ and on $C$ and preserves $F$, we have  
$$\tau|_C \circ \sigma (\tau|_C)^{-1}=\sigma^{-1}=\sigma^{m-1}.$$

\begin{lemma}  \label{le:Ktsigmatoitself3}
    Let $\tau\in {\rm Isom}(D,D^{op})$, $k\in \mathbb{N}$, and $\alpha\in C((x))^\times$.  The map generated by $\widetilde G(c)=\tau(c)$ for all
    $c\in D((t))$ and $\widetilde  G(t)=\alpha t^k$ defines an anti-homomorphism of rings on $D((t;\sigma))$, if any only if  $\tau\sigma  \tau^{-1}=\sigma^k$.
     In particular, if $\tau$ and $\sigma$ commute then this occurs if and only if $k\equiv 1 \mod m$.
This map is $F^\tau$-linear. We denote it by $\widetilde G_{\tilde \tau,\alpha,k}$.
\end{lemma}

\begin{proof}
We observe that every  $\tau\in {\rm Isom}(D,D^{op})$ canonically induces an isomorphism  $\tau:D((t))\to D((t))^{op}$ that leaves $t$ invariant.
The map $\widetilde  G:D((t;\sigma)) \to D((t;\sigma))$ generated by $\widetilde G(c)=\tau(c)$ for all $c\in D((t))$ and $\widetilde  G(t)=(\alpha t^{k})^{-1}$ is given by
$$
\widetilde  G(\sum_{i} a_i t^i) = \sum_{i} (\alpha t^k)^{-i}   \tau(a_i) = \sum_i  N_i^{\sigma^{-k}}(\alpha) t^{-ik}\tau(a_i)
$$
for any $a_i\in D((t))$.
 It will be an anti-homomorphism if and only if  the canonically induced  map $G$ between $ D((t;\sigma))$ and its opposite algebra $ D^{op}((t;\sigma^{-1}))$ is multiplicative which is the case, if and only if  $G(t)G(c)=G(\sigma(c))G(t)$ for all $c\in D$.
 We compute
$$
G(t)G(c)=\alpha t^{k} \tau(c) = \alpha \sigma^{k}(\tau(c)) t^{k}
$$
and
$$
G(\sigma(c))G(t)=\tau(\sigma(c))\alpha t^{k}.
$$
This holds for all $c\in D$ if and only if $\tau\sigma=\sigma^{k}\tau$.
 In particular, if $\sigma$ and $\tau$ commute, we must have $\sigma=\sigma^{k}$, or  $k\equiv 1 \mod m$.
\end{proof}

Assume that $\tau:D\to D^{op}$ is a ring isomorphism between $D$ and $D^{op}$, viewed as algebras over some subfield $F_0$ of $F$. Let  $\alpha\in C((x))^\times$ and $F_0=F^\tau=\{a\in F\,|\,\tau(a)=a \}$.

An element in $(D((x)),\sigma, ux)=D((x))[z,\sigma]/D((x))[z;\sigma](z^m-ux)$ has the form
 $$ \sum_{i=0}^{m-1}(\sum_{j_i}c_{i,j_i} x^{j_i}) z^i$$
 with $c_{i,j}\in D$ and the canonical algebra isomorphism between $(D((x)),\sigma, ux)$ and $D((t;\sigma))$ is given by
 $$ \sum_{i=0}^{m-1}(\sum_{j_i}c_{i,j_i} x^{j_i}) z^i\mapsto  \sum_{i=0}^{m-1}(\sum_{j_i}c_{i,j_i} t^{mj_i}) t^i.$$
By Theorem   \ref{thm:maingeneralizedantiaut2}, there exists a $\tau$-semilinear anti-automorphism $\widetilde G_{\tau,\alpha,m-1}$ on $ (D((x)),\sigma,ux)$,
 $$\widetilde G_{\tau,\alpha,m-1} ( \sum_{i=0}^{m-1}(\sum_{j_i}c_{i,{j_i}} x^{j_i}) z^i ) =
\sum_{i=0}^{m-1}  (\alpha z^{m-1})^{-i}(\sum_{j_i} x^{j_i} \tau(c_{i,{j_i}}))$$
  if and only if  $\tau\sigma=\sigma\tau$ and
 $$N_{C((x))/F((x))}(\alpha)=\tau(ux)(ux)^{m-1}.$$
 Observe that
 $\sigma(u) = u$ and so the right-hand side of this norm equation becomes
 $$\tau(ux)(ux)^{m-1}=\tau(u)x N_{m-1}^{\sigma^m}(u)x^{m-1}=\tau(u) N_{m-1}^{\sigma^m}(u) x^m=\tau(u) u^{m-1} x^m.$$
 In particular, when $\sigma$ and $\tau$ commute on $D((x))$ then $k=1$, $m=2$, and there exists a $\tau$-semilinear anti-automorphism $\widetilde G_{\tilde\tau,\alpha,1}$ on $(D((x)),\sigma, ux)$ defined via the isomorphism
$$G_{\tau,\alpha,1} ( \sum_{i=0}^{1}(\sum_{j_i}c_{i,{j_i}} x^{j_i}) z^i ) =
\sum_{i=0}^{1}(\sum_{j_i}\tau(c_{i,{j_i}}) x^{j_i}) (\alpha z)^i$$
 if and only if $D\cong D^{op} $ as $F_0$-algebras via the isomorphism $\tau$ (which is equivalent to requiring that $D((x))\cong D((x))^{op}$), and there exists  some $\alpha\in C((x))^\times$ such that
$$N_{C((x))/F((x))}(\alpha)=\tau(ux)ux=\tau(u)\sigma^2(u)x^2 =\tau(u)ux^2.$$

 We compute explicitly that
  $$
 G_{ \tau,\alpha,k} ( \sum_{i=0}^{m-1}(\sum_{j_i}c_{i,{j_i}} x^{j_i}) z^i ) =
\sum_{i=0}^{m-1}(\sum_{j_i}\tau(c_{i,{j_i}}) x^{j_i}) (\alpha z^k)^i =
\sum_{i=0}^{m-1}(\sum_{j_i}\tau(c_{i,{j_i}}) x^{j_i})  N_i^{\sigma^k}(\alpha) z^{ik}
$$
$$ = \sum_{i=0}^{m-1}(\sum_{j_i}\tau(c_{i,{j_i}}) x^{{j_i}}  N_i^{\sigma^k}(\alpha) z^{ik}
=  \sum_{i=0}^{m-1}\sum_{j_i}\tau(c_{i,{j_i}}) \sigma^{m{j_i}}( N_i^{\sigma^k}(\alpha)) x^{j_i}  z^{ik}.$$

\begin{theorem}\label{thm:maingeneralizedantiaut6}
Let $\tau:D\to D^{op}$ be a ring isomorphism and put $F_0=F^\tau$. Let  $\alpha\in C((x))^\times$.
\\ (i) Suppose that $\tau\sigma=\sigma\tau$, that means $m=2$. Then there exists a $\tau$-semilinear anti-automorphism $\widetilde G_{\tilde\tau,\alpha,1}$ on  the ring of skew Laurent polynomials  $(D((x)),\sigma, ux)$ defined via the isomorphism
$$ G_{\tau,\alpha,1} ( \sum_{i=0}^{m-1}(\sum_{j_i}c_{i,{j_i}} x^{j_i}) z^i ) =\sum_{i=0}^{m-1}\sum_{j_i}\tau(c_{i,{j_i}}) \sigma^{m{j_i}}( N_i^{\sigma}(\alpha)) x^{j_i}  z^{i}$$
 if and only if $D\cong D^{op} $ as $F_0$-algebras via $\tau$, and there exists  some $\alpha\in C((x))^\times$ such that
$$N_{C((x))/F((x))}(\alpha)= \tau(u) ux^2.$$
 (ii) For any $2\leq k < m$, the map
$\widetilde  G_{\tilde\tau,\alpha,m-1}$ associatited with the isomorphism
$$G_{\tau,\alpha,k} ( \sum_{i=0}^{m-1}(\sum_{j_i}c_{i,{j_i}} x^{j_i}) z^i ) =\sum_{i=0}^{m-1}\sum_{j_i}\tau(c_{i,{j_i}}) \sigma^{m{j_i}}( N_i^{\sigma^k}(\alpha)) x^{j_i}  z^{ik},$$
 is a $\tau$-semilinear anti-automorphism on  $(D((x)),\sigma, ux)$, if and only if
 $ \tau\sigma \tau^{-1}=\sigma^{m-1}$ and
  $$N_{C((x))/F((x))}(\alpha)=\tau(u)u^{m-1} x^m.$$
\end{theorem}

This follows from Theorem \ref{thm:maingeneralizedantiautcor} and the fact that we observed that $k=m-1$ here.

Suppose $\sigma$ has order $m$, thus $u=1$. We have  $N_{C((x))/F((x))}(x)= x^m$ and obtain the following result.

\begin{corollary} \label{cor:maingeneralizedantiaut6}
Suppose $\sigma$ has order $m$ and that $\tau:D\to D^{op}$ is a ring isomorphism. Put $F_0=F^\tau$.
\\ (i) Suppose that $\tau\sigma=\sigma\tau$, that means $m=2$. Then there exists a $\tau$-semilinear anti-automorphism $ G_{\tau,x,1}$ on $(D((x)),\sigma, x)$ canonically defined via the isomorphism
$$ G_{\tau,x,1} ( \sum_{i=0}^{m-1}(\sum_{j_i}c_{i,{j_i}} x^{j_i}) z^i ) =\sum_{i=0}^{m-1}\sum_{j_i}\tau(c_{i,{j_i}})  x^{i-1} x^{j_i}  z^{i}$$
 if and only if $D\cong D^{op} $ as $F_0$-algebras via $\tau$.
 \\ (ii) For any $2\leq k < m$, the isomorphism
$ G_{\tau,x,m-1}$,
$$ G_{\tau,x,k} ( \sum_{i=0}^{m-1}(\sum_{j_i}c_{i,{j_i}} x^{j_i}) z^i ) =\sum_{i=0}^{m-1}\sum_{j_i}\tau(c_{i,{j_i}}) x^{i-1} x^{j_i} z^{ik},$$
 canonically defines a $\tau$-semilinear anti-automorphism on  $(D((x)),\sigma, ux)$, if and only if
 $ \tau\sigma \tau^{-1}=\sigma^{m-1}$.
\end{corollary}

\section{Outlook}\label{sec:Azumaya}

Since the theory we developed is bases on skew polynomials it can be canonically extended to cyclic and generalized cyclic Azumaya algebras over rings.

Let $D$ be an Azumaya algebra of constant rank $r$ with center $C$.
Let $\sigma\in {\rm Aut}(D)$ be a ring automorphism, such that $\sigma|_{C}$ has finite order $m$ and fixed ring $S={\rm Fix}(\sigma)\cap C$. We will assume that $C/S$ is a cyclic Galois ring extension of degree $m$ with Galois group $\mathrm{Gal}(C/S) = \langle \sigma |_{C} \rangle$. This implies that $C$ has constant rank $m$ as an $S$-module.

Let  $f(t)=t^m-d\in D[t;\sigma]$ with $d\in D^\times$ invertible.  Then $(D,\sigma,d)=D[t;\sigma]/D[t;\sigma](t^m-d)$ is a nonassociative (Petit) algebra over its center
$S=C\cap{\rm Fix}(\sigma);$
and an associative Azumaya algebra over $S$  of constant rank $rm$ if and only if $d\in S^\times$ which is called a  \emph{nonassociative generalized  cyclic Azumaya algebra} when $d\in S^\times$.

If $D=C$ is a commutative ring, $C/S$ is a cyclic Galois extension of rings of degree $m$ with Galois group generated by $\sigma$, and $ d\in S$. In this case $(C/S_0,\sigma,d)=C[t,\sigma]/C[t;\sigma](t^m-d)$ is an Azumaya algebra which is called an associative \emph{cyclic Azumaya algebra}  of rank $m^2$  \cite{Pum2021, Sz}.

It is straightforward to check that the results from the last section also hold for (generalized) cyclic Azumaya algebras, since for the right choice of $\alpha\in C^\times$ and $\tau\in \Aut(D)$, the correspondingly defined maps $G_{\tau,\alpha,k}$ and $\widetilde G_{\tau\tau,\alpha,k}$ are also isomorphism/anti-automorphisms for generalized Azumaya algebras.

This broadly ties in with problems studied in \cite{First}, where Azumaya algebras without involution are constructed; and examples like  the one of an an Azumaya algebra $B$ over a Dedekind domain $C$ such that $B\otimes_CB $ is Brauer-equivalent to $C$ are given, where $ B$ is not Brauer-equivalent to $C$ and $B$ does not have a $C$-anti-automorphism.


\begin{thebibliography}{1}

\bibitem[Albert1939]{A} A. A. Albert, Structure of algebras. Vol. 24, AMS 1939.

\bibitem[Brown2018]{CB} C. Brown, \emph{Petit's algebras and their  automorphisms.} PhD Thesis, University of Nottingham, 2018.

\bibitem[BrownPu2019]{BrownPumpluen2019} C. Brown,  S. Pumpl\"un, \emph{Solvable crossed product algebras revisited}.
 Glasgow Mathematical Journal (April 2019),
 https://doi.org/10.1017/S0017089519000089

 \bibitem[BrownPu2018]{BP}  C. Brown,  S. Pumpl\"un, \emph{The automorphisms of Petit's algebras}.
  Comm. Algebra  46 (2) (2018), 834-849.

\bibitem[First2015]{First}  U. A. First,
\emph{Rings that are Morita equivalent to their opposites}
J. Algebra 430, (2015), 26-61.

\bibitem[Hanke2001]{TH} T. Hanke, ``A Direct Approach to Noncrossed Product Division Algebras.'' Dissertation 2001, Universit\"at Potsdam, Naturwissenschaftliche Fakult\"at.
\verb#https://doi.org/10.48550/arXiv.1109.1580#

\bibitem[Hanke2004]{TH2004}  T. Hanke, \emph{An explicit example of a noncrossed product division algebra}.
Math. Nachr. 271 (2004), 51-68.


\bibitem[Hanke2005]{TH2005}  T. Hanke, \emph{A twisted Laurent series ring that is a noncrossed product}.
   Israel J. Math. 15  (2005), 199-203.

\bibitem[Hanke2007]{TH2007}  T. Hanke, \emph{The isomorphism problem for cyclic algebras and an  application}. ISSAC 2007, 181-186, ACM, New York, 2007.


\bibitem[Jac1996]{J96} N.~Jacobson,
``Finite-dimensional division algebras over fields.'' Springer Verlag,
Berlin-Heidelberg-New York, 1996.

\bibitem[KMRT1998]{KMRT} Knus, M.A., Merkurjev, A., Rost, M., Tignol, J.-P.,
``The Book of Involutions'', AMS Coll. Publications, vol. 44 (1998).

\bibitem[Lew2006]{Lewis} D. W. Lewis, \emph{Involutions and anti-automorphisms of algebras}. Bull. London Math. Soc. 38 (2006), 529-546. 

\bibitem[Lew2003]{Lewis2} D. W. Lewis, \emph{Involutions and anti-automorphisms of central simple algebras}. J. Pure Appl. Algebra 182 (2003), 253-261.

\bibitem[LewTig1999]{LewisTig} D. W. Lewis, \emph{Classification theorems for central simple algebras with involutions}. Manuscripta Math. 100 (1999), 259-276.

\bibitem[MorST2005]{MST} P. J. Morandi, B. A. Sethuraman, J.-P. Tignol, \emph{Division algebras with an anti-automorphism but with no involution}
  Adv. Geom. 5 (2005), 485-495. \verb#https://doi.org/10.1515/advg.2005.5.3.485#

  \bibitem[NevPum2025]{NevPum2025} M. Nevins, S. Pumpl\"un, \emph{When isometry and equivalence for skew constacyclic codes coincide.}
Preprint 2025.
\verb#arXiv:2508.06695v1# [cs.IT]

\bibitem[Petit1966]{P66} J.-C. Petit, \emph{Sur certains quasi-corps g\'{e}n\'{e}ralisant un type d'anneau-quotient}.
S\'{e}minaire Dubriel. Alg\`{e}bre et th\'{e}orie des nombres 20 (1966 - 67), 1-18.

  \bibitem[Pum2025]{Pum2025}
Susanne Pumpl\"un, \emph{Using nonassociative algebras to classify skew polycyclic codes up to isometry and equivalence},
\verb#https://doi.org/10.48550/arXiv.2508.10139# [cs.IT]

 \bibitem[Pum2021]{Pum2021} \emph{The automorphisms of generalized cyclic Azumaya algebras.}
J. Pure Applied Algebra 225 (4) (2021). \\
\verb#https://doi.org/10.1016/j.jpaa.2020.106540#

\bibitem[Rim78]{Rimmer1978}
M.~Rimmer, \emph{Isomorphisms between skew polynomial rings}, J. Austral. Math. Soc. Ser. A \textbf{25} (1978), no.~3, 314--321. \MR{491835}




\bibitem[Steele2014]{S12} A.~Steele, \emph{ Nonassociative cyclic algebras.}
 Israel Journal of Mathematics 200 (1) (2014),  361-387.

 \bibitem[Szeto1982]{Sz} G.~Szeto, \emph{Splitting rings for Azumaya quaternion algebras.} Brauer groups in ring theory and algebraic geometry (Wilrijk, 1981),  Lecture Notes in Math., 917,  118-125,   Springer, Berlin-New York, 1982.

\bibitem[Tignol1987]{T} J.-P. Tignol, \emph{Generalized crossed products.}  S\'{e}minaire Math\'{e}matique (nouvelle
s\'{e}rie) 106, Universit\'{e} Catholique de Louvain, Louvain-la-Neuve, Belgium,
1987.

\end{thebibliography}
\end{document}